\documentclass[reqno]{amsart}
\usepackage{CJK}
\usepackage{hyperref}
\usepackage{cite}
\usepackage{amsmath}
\usepackage{mathrsfs}
\usepackage{bm}
\usepackage{graphicx}
\usepackage{epsfig}
\usepackage{epstopdf}

\numberwithin{equation}{section}

\newtheorem{theorem}{Theorem}[section]
\newtheorem{lemma}[theorem]{Lemma}
\newtheorem{definition}[theorem]{Definition}

\newtheorem{proposition}[theorem]{Proposition}
\newtheorem{remark}[theorem]{Remark}
\newtheorem{corollary}[theorem]{Corollary}

\allowdisplaybreaks

\makeatletter
\@namedef{subjclassname@2020}{%
	\textup{2020} Mathematics Subject Classification}
\makeatother

\newcommand{ \mint }{ {\int\hspace{-0.38cm}-}}

\begin{document}
	
\title[\hfil Harnack inequality for doubly nonlinear parabolic equations] {Harnack inequality for doubly nonlinear mixed local and nonlocal parabolic equations}

\author[V.D. R\u adulescu, B. Shang, C. Zhang \hfil \hfilneg]
{Vicen\c tiu D. R\u adulescu, Bin Shang, Chao Zhang$^*$}

\thanks{ORCID: 0000-0003-4615-5537 (Vicen\c{t}iu D. R\u{a}dulescu), 0000-0003-2702-2050 (Chao Zhang)} 

\thanks{$^*$Corresponding author.}

\address{Vicen$\c{t}$iu D. R\u{a}dulescu \hfill\break 
	Faculty of Applied Mathematics, AGH University of Krak\'ow, 30-059, Krak\'ow, Poland \&	Department of Mathematics, University of Craiova, Street A.I. Cuza 13, 200585 Craiova, Romania} \email{radulescu@inf.ucv.ro}

\address{Bin Shang \hfill\break
	School of Mathematics, Harbin Institute of Technology,
	Harbin 150001, P.R. China \&
	Department of Mathematics, University of Craiova, Street A.I. Cuza 13, 200585 Craiova, Romania
} \email{shangbin0521@163.com}

\address{Chao Zhang\hfill\break
	School of Mathematics and Institute for Advanced Study in Mathematics, Harbin Institute of Technology,
	Harbin 150001, P.R. China} \email{czhangmath@hit.edu.cn}


\maketitle

\begin{abstract}
In this paper, we establish the Harnack inequality of nonnegative weak solutions to the doubly nonlinear mixed local and nonlocal parabolic equations. This result is obtained by combining  a related comparison principle, a local boundedness estimate, and an integral Harnack-type inequality. Our proof is based on the expansion of positivity together with a comparison argument. 
\end{abstract}

\section{Introduction}
\label{sec1}
\par

Let $E_T:=E\times(0,T)$ be an open bounded set $E\subset \mathbb{R}^N$ and $T>0$. In this paper, we discuss the Harnack-type estimate for nonnegative weak solutions to the following doubly nonlinear parabolic equation 
\begin{align}
\label{1.1}
\partial_t \left(|u|^{q-1}u\right)-\mathrm{div}(|\nabla u|^{p-2}\nabla u)+\mathcal{L}u=0 \quad \text{in } E_T,
\end{align}
where $p>1$, $q>0$ and the operator $\mathcal{L}$ is given by
\begin{align}
\label{1.2}
\mathcal{L} u(x,t)=\mathrm{P.V.} \int_{\mathbb{R}^N} K(x,y,t)|u(x,t)-u(y,t)|^{p-2}(u(x,t)-u(y,t))\,dy.
\end{align}
Here, P.V. means the Cauchy principal value and $K(x, y, t):\mathbb{R}^N\times\mathbb{R}^N\times (0,T]\rightarrow [0,\infty)$ is the symmetric kernel function satisfying
\begin{align*}
\frac{\Lambda^{-1}}{|x-y|^{N+sp}} \leq K(x,y,t)\equiv K(y,x,t) \leq \frac{\Lambda}{|x-y|^{N+sp}}\quad \mathrm{a.e. }\ x,y\in\mathbb{R}^N
\end{align*}
for some $\Lambda \geq 1$ and $s \in(0,1)$. 

\smallskip
To state the definition of weak solutions, we denote by $W_0^{1,p}(E)$ the Sobolev space with zero boundary values, namely 
\begin{align*}
W_0^{1,p}(E):=\left\{u \in W^{1,p}(E): u=0 \text { in } \mathbb{R}^N \backslash E \right\}.
\end{align*}
Moreover, we introduce the tail space
\begin{align*}
L_\alpha^m(\mathbb{R}^{N}):=\left\{v \in L_{\rm{loc}}^m(\mathbb{R}^{N}): \int_{\mathbb{R}^N} \frac{|v(x)|^m}{1+|x|^{N+\alpha}}\,dx<+\infty\right\}, \quad m>0 \text{ and } \alpha>0.
\end{align*}
The parabolic nonlocal tail is of the form
\begin{align*}
\mathrm{Tail}_\infty(v; x_0, R; t_0-S, t_0):=\operatorname*{ess\,\sup}_{t_0-S <t<t_0}\left(R^p \int_{\mathbb{R}^{N} \backslash B_R(x_0)} \frac{|v(x, t)|^{p-1}}{|x-x_0|^{N+sp}}\,dx\right)^{\frac{1}{p-1}}.
\end{align*}
Note that $\mathrm{Tail}_\infty(v;x_0,R;t_0-S,t_0)$ is well-defined for any $v\in L^\infty(t_0-S,t_0;L_{sp}^{p-1}(\mathbb{R}^N))$.

Eq. \eqref{1.1} can be seen as a mixed version of the doubly nonlinear parabolic equation
\begin{align}
\label{1.4}
\partial_t \left(|u|^{q-1}u\right)-\mathrm{div}(|\nabla u|^{p-2}\nabla u)=0 \quad \text{in } E_T,
\end{align}
which attracts lots of interest both in light of its mathematical structure and its significance to describe many physical phenomena including shallow water flows \cite{ASD08} and glacier dynamics \cite{M76}, see also \cite{V06, LM18} for other classical applications. It is well-known that Eq. \eqref{1.4} covers Trudinger's equation $(q=p-1)$, the evolutionary $p$-Laplace equation $(q=1)$, and the porous medium equation $(p=2)$. There have been fruitful achievements on Eq. \eqref{1.4}, such as the existence of solutions \cite{BDMS18, MN23}, H\"{o}lder regularity \cite{BDL21, V92}, higher integrability \cite{BDKS20, MSS23}, and Harnack type estimate \cite{V94,BDG23, BHS21, VV22}. We refer the readers to \cite{ GV06, KK07, MN23-1, M23} and reference therein for more related results.

For the doubly nonlinear nonlocal parabolic equation
\begin{align}
\label{1.5}
\partial_t \left(|u|^{q-1}u\right)+(-\Delta)_p^s u=0 \quad \text{in } E_T,
\end{align}
the pointwise behavior of weak solutions was shown in \cite{BGK23} with $p>1$ and $q>0$. For what concerns the nonlocal Trudinger equation, Banerjee-Garain-Kinnunen \cite{BGK22} investigated the local boundedness of weak subsolutions by De Giorgi's method. Meanwhile, they also provide a crucial algebraic inequality to deal with the nonlocal term in obtaining a reverse H\"{o}lder inequality for positive weak supersolutions under the assumption that $p>2$. In particular, a weak Harnack inequality for globally bounded positive weak solutions to the nonlocal Trudinger equation was derived by Prasad \cite{P24}. Taking into account the fractional $p$-Laplace parabolic equation, the local boundedness $(p>1)$ and the H\"{o}lder regularity $(p>2)$ for local weak solutions were explored in \cite{DZZ21}. Moreover, \cite{L24} generalized the regularity results to the region $1<p<\infty$ by means of the intrinsic scaling method, but avoids using any comparison principle. Further results can be found in \cite{S19, V16, BLS21}.

Regarding the homogeneous scenario of \eqref{1.1}, i.e., $q=p-1$, the local boundedness of sign-changing weak solutions for $p\geq 2$ was achieved by Nakamura \cite{N22}. Later, the author considered Harnack's inequality for globally bounded positive weak solutions through some quantitative estimates in \cite{N23}. When $q=1$, Eq. \eqref{1.1} becomes the mixed local and nonlocal parabolic $p$-Laplace equation
\begin{align}
\label{1.6}
\partial_t u-\mathrm{div}(|\nabla u|^{p-2}\nabla u)+\mathcal{L}u=0 \quad \text{in } E_T,
\end{align}
some regularity properties of sign-changing weak solutions involving the local boundedness, the semicontinuity, and the pointwise behavior were discussed in \cite{GK24}. By employing the expansion of positivity, Shang-Zhang studied Harnack's estimate and the H\"{o}lder continuity for weak solutions to \eqref{1.6} in \cite{SZ23} and \cite{SZ22}, respectively. Very recently, Adimurthi-Prasad-Tewary \cite{APT23} developed the $C^{1, \alpha}$ regularity of weak solutions to \eqref{1.6} by suitable comparison estimates. 

To the best of our knowledge, there are no results concerning Eq. \eqref{1.1} for the general case $q\neq p-1$. Motivated by the works \cite{BDG23, N23}, our purpose in this paper is to study the Harnack inequality of globally bounded nonnegative weak solutions to problem \eqref{1.1} under the assumptions that  $0<p-1<q<p^2-1$ and $p>N$. Different from the homogeneous case, the solutions to Eq. \eqref{1.1} in the fast diffusion range can not be scaled by multiplying any scale factor. Exactly for this, we adopt the approach called the expansion of positivity and choose suitable geometries to overcome the non-homogeneity. In addition, we remark that the assumptions proposed on $p$ and $q$ address another difficulty stemming from the nonlocal feature. To investigate the Harnack inequality, we also establish the local boundedness, a comparison principle, and an integral-type Harnack inequality as byproducts.

\smallskip

We now state the notion of weak solutions to problem \eqref{1.1}.

\begin{definition}
We identify a function $u$ is a weak subsolution (super-) to the doubly nonlinear mixed local and nonlocal parabolic equation \eqref{1.1} if 
\begin{align*}
u\in C(0,T; L^{q+1}(E))\cap L^p(0,T;W^{1,p}(E))\cap L^\infty(0,T;L_{sp}^{p-1}(\mathbb{R}^N))
\end{align*}
such that there holds the integral inequality
\begin{align}
\label{1.7}
\iint_{E_T}-|u|^{q-1} u \partial_t \varphi+|\nabla u|^{p-2}\nabla u \cdot \nabla \varphi\, dxdt+\int_{0}^{T} \mathcal{E}(u, \varphi, t)\,dt \leq (\geq) 0
\end{align}
for all nonnegative test functions $\varphi \in W_0^{1, q+1}(0,T;L^{q+1}(E))\cap L^p(0,T;W_0^{1,p}(E))$, where
\begin{align*}
\mathcal{E}(u,\varphi,t):=\int_{\mathbb{R}^N}\int_{\mathbb{R}^N} K(x,y,t)|u(x,t)-u(y,t)|^{p-2}(u(x,t)-u(y,t))(\varphi(x,t)-\varphi(y,t))\,dydx.
\end{align*}
A function $u$ is called a weak solution to \eqref{1.1} if it is both a weak subsolution and a weak supersolution.
\end{definition}

At this stage, we present our results as follows. The first one is about the local boundedness of weak subsolutions.
\begin{theorem}
\label{thm-1-2}
Let $0<p-1<q$ and let $r\geq 1$ such that 
\begin{align}
\label{1.8}
\lambda_r:=N(p-q-1)+rp>0.
\end{align}
Suppose that $u$ is a nonnegative, weak subsolution to \eqref{1.1}. For $r>m:=p\frac{N+q+1}{N}$, we further assume that $u$ is qualitatively locally bounded. Let $Q_{\rho,s}=K_\rho(x_0)\times(t_0-s,t_0] \subset\subset E_T$. Then there holds
\begin{align*}
\underset{Q_{\frac{1}{2} \rho, \frac{1}{2} s}}{\operatorname{ess} \sup } \ u \leq & \gamma\left(\frac{\rho^p}{s}\right)^{\frac{N}{\lambda_r}}\left(\mint\!\!\mint_{Q_{\rho,s}} u^r\,dxdt\right)^{\frac{p}{\lambda_r}}+\gamma\left(\frac{s}{\rho^p}\right)^{\frac{1}{q+1-p}}\\
&+\gamma\left(\frac{s}{\rho^p}[\mathrm{Tail}_\infty(u; x_0, \rho/2; t_0-s, t_0)]^{p-1}\right)^{\frac{1}{q}},
\end{align*}
where $\gamma$ is a constant depending only on $N,p,s,q,\Lambda$.	
\end{theorem}
\begin{remark}
\label{rem-1-3}
We will discuss the local boundedness result in three cases. The first two cases care about the exponent $r\leq m$, in which $u\in L_{\rm{loc}}^m(E_T)$ can be ensured by Sobolev embedding in Lemma \ref{lem-2-7}. For $r>m$, we further assume that the weak subsolutions of \eqref{1.1} are locally bounded because Lemma \ref{lem-2-7} is generally not true. Notice that there exists some $r\geq 1$ satisfying $r\leq m$ and $\lambda_r>0$, if and only if $\lambda_m>0$. Owing to $$\lambda_m=(N+p)(m-q-1)=\frac{N+p}{N}\lambda_{q+1},$$ $\lambda_m>0$ asking for $m>q+1$ or $\lambda_{q+1}>0$, which is equal to
\begin{align*}
q<\frac{N(p-1)+p}{(N-p)_+}.
\end{align*}
Consequently, we can apply Theorem \ref{thm-1-2} for $r=q+1$ without assuming a prior that $u$ is locally bounded in this case.
\end{remark}

From Remark \ref{rem-1-3}, we get a corollary of Theorem \ref{thm-1-2}.
\begin{corollary}
\label{cor-1-4}
Let $0<p-1<q<\frac{N(p-1)+p}{(N-p)_+}$. Then we know that every nonnegative weak subsolution to \eqref{1.1} is locally bounded.
\end{corollary}

Based on the boundedness result, we can derive the Harnack inequality in an integral-type, which is a key ingredient to obtain the Harnack inequality.

\begin{theorem}
\label{thm-1-5}
Let $0<p-1<q<\min\left\{p^2-1,\frac{N(p-1)}{(N-p)_{+}}\right\}$ and $\lambda_q:=N(p-1-q)+q p>0$. Suppose that $u$ is a nonnegative, weak solution to \eqref{1.1}, and the cylinder $Q_{\rho,s}=K_\rho(x_0)\times(t_0-s,t_0]\subset E_T$ with $\rho\in(0,1]$. Then there holds
\begin{align*}
\operatorname*{ess \sup}_{Q_{\frac{1}{2} \rho, \frac{1}{2} s}} u \leq &\gamma\left(\frac{\rho^p}{s}\right)^{\frac{N}{\lambda q}}\left(\inf _{t \in\left[t_0-s,t_0\right]} \mint_{K_\rho(x_0) \times\{t\}} u^q \,dx \right)^{\frac{p}{\lambda_q}}+\gamma\left(\frac{s}{\rho^p}\right)^{\frac{1}{q+1-p}}\\
&+\gamma\left(\frac{s}{\rho^p}[\mathrm{Tail}_\infty(u;x_0,\rho/2;t_0-s,t_0)]^{p-1}\right)^{\frac{1}{q}}\\
&+\gamma \left(\frac{s}{\rho^p}\right)^{\frac{p-N}{\lambda_q}}\left[\mathrm{Tail}_\infty(u;x_0,\rho/2;t_0-s,t_0)\right]^{\frac{p(p-1)}{\lambda_q}}
\end{align*}
with constant $\gamma$ depending only on $N, p, s, q, \Lambda$.
\end{theorem}

The last theorem is our main result regarding the Harnack inequality of nonnegative weak solutions while we additionally assume that the weak solutions are globally bounded.
\begin{theorem}
\label{thm-1-6}
Let $p>N$, $0<p-1<q<p^2-1$ and let $\rho\in(0,1]$. Suppose that $u\in L^\infty(\mathbb {R}^N\times(0,T))$ is a nonnegative, continuous, weak solution to \eqref{1.1}, and $u(x_0,t_0)>0$. There exist constants $\gamma>1$ and $\sigma \in(0,1)$ depending only on $N, p, s, q, \Lambda, \|u\|_{L^\infty(\mathbb{R}^N\times(0,T))}$, such that for any
\begin{align*}
(x, t) \in K_{\rho}\left(x_0\right) \times\left(t_0-\sigma\left[u\left(x_0, t_0\right)\right]^{q+1-p} \rho^p, t_0+\sigma\left[u\left(x_0, t_0\right)\right]^{q+1-p} \rho^p\right),
\end{align*}
we have
\begin{align*}
\gamma^{-1} u\left(x_0, t_0\right) \leq u(x, t) \leq \gamma u\left(x_0, t_0\right),
\end{align*}
provided 
\begin{align*}
K_{8 \rho}\left(x_0\right) \times\left(t_0-\gamma[u(x_0,t_0)]^{q+1-p}(8 \rho)^p, t_0+\gamma[u(x_0,t_0)]^{q+1-p}(8 \rho)^p\right) \subset E_T.
\end{align*}
\end{theorem}
\begin{figure}[htbp] 
	\centering 
	\includegraphics[width=0.6\textwidth]{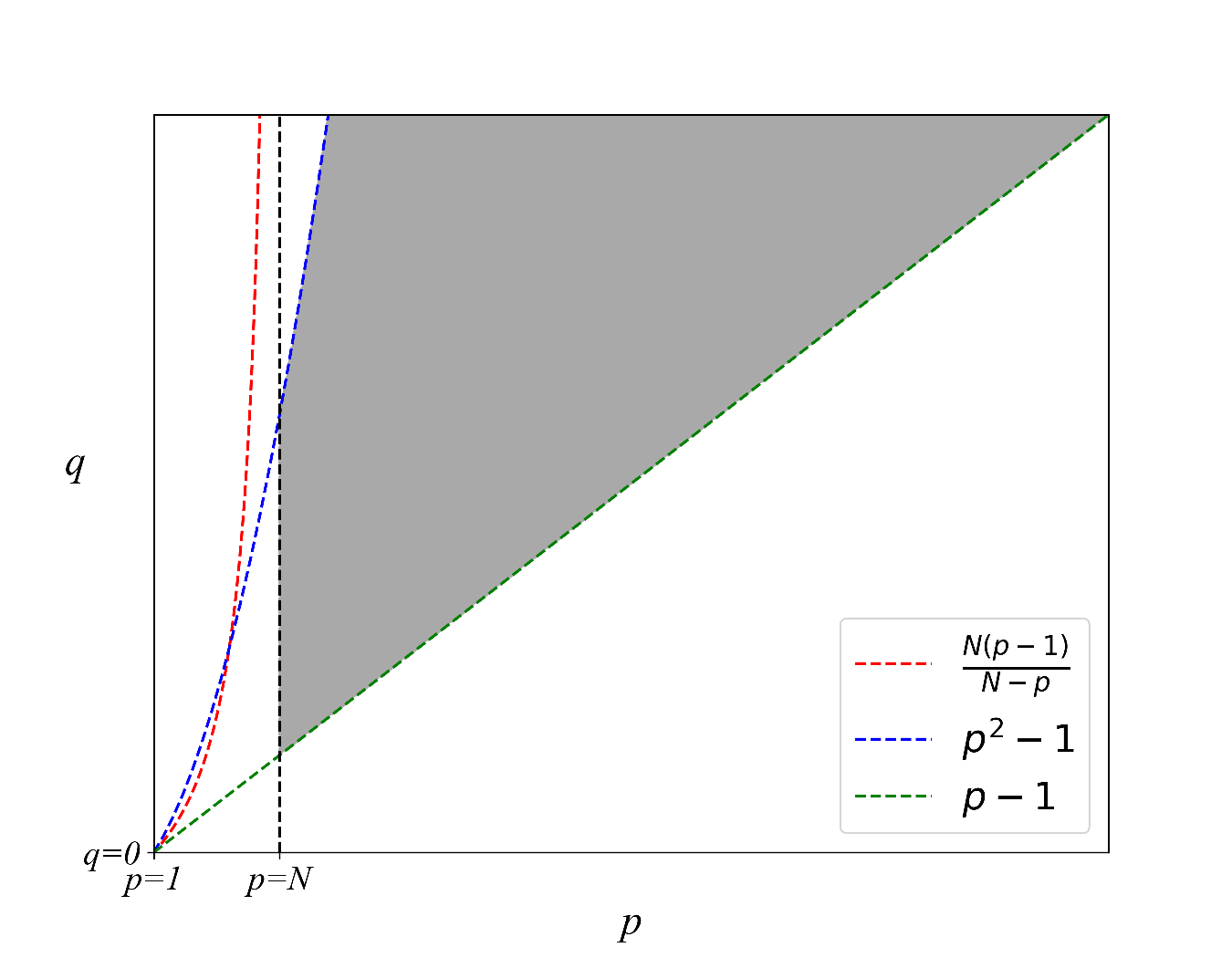} 
	\caption{} 
	\label{Fig} 
\end{figure}

The range where Theorem \ref{thm-1-6} holds is shown in Figure \ref{Fig}. The Harnack inequality we established above distinguishes from the normal parabolic Harnack inequality in two aspects. First, we obtain the pointwise information in an intrinsic cylinder since we perform a specific scaling of the equation. Second, the effect of time in this Harnack inequality is weakened, so that it presents a ``elliptic" feature.

\begin{remark}
\label{rem-1-7}
In the statement of Theorem \ref{thm-1-6}, we assume that $u$ is a continuous function for giving a clear sense to $u(x_0,t_0)$. In fact, this property can be proved by using De Giorgi-type Lemma \ref{lem-6-2} in the forthcoming context together with Theorem 2.1 in \cite{L21}, and the detailed proof can be found in \cite{BGK23}.
\end{remark}

This paper is organized as follows. In Section \ref{sec2}, we display some basic notations and give several preliminary materials. Section \ref{sec3} is devoted to deriving a comparison principle. In Section \ref{sec4}, we will provide the Caccioppoli-type inequality first, and then discuss the local boundedness result Theorem \ref{thm-1-2}. The integral-type Harnack inequality Theorem \ref{thm-1-5} will be proved in Section \ref{sec5}. In Section \ref{sec6}, we will develop the expansion of positivity of weak solutions. Finally, we complete the proof of main result (Theorem \ref{thm-1-6}) in Section \ref{sec7}.


\section{Preliminaries}
\label{sec2}

\subsection{Notation}
First, we collect some notations used throughout the paper. We shall denote $K_\rho(x_0)$ is a cube centered at $x_0\in\mathbb{R}^N$, whose side length $2\rho>0$, and whose faces are parallel to the coordinate planes in $\mathbb{R}^N$. As is customary, we write the general backward parabolic cylinders as
\begin{align*}
(x_0, t_0)+Q_{R,S}:=K_R(x_0) \times\left(t_0-S,t_0\right].
\end{align*}
We will omit $(x_0,t_0)$ if the context is clear or $(x_0,t_0)=(0,0)$.

For fixed $k\in \mathbb{R}$, define
\begin{align*}
(u-k)_+=\max\{u-k,0\} \quad \textmd{and} \quad  (u-k)_-=\max\{-(u-k),0\}.
\end{align*} 
For a function $u$ defined in $E$ and a real number $l$, we denote
$$[u>l]=\{x\in E:u(x)>l\}.$$
We also use the shorthand notations
\begin{align*}
d\mu=d\mu(x,y,t)=K(x,y,t)\,dxdy 
\end{align*}
and
\begin{align*}
U(x,y,t):=|u(x,t)-u(y,t)|^{p-2}(u(x,t)-u(y,t)).
\end{align*}
We denote by $\gamma$ some generic constants, that may vary from each other even in the same line.
\subsection{Technical lemmas}
We give an algebraic inequality, see Lemma 2.2 in \cite{AF89} for $0<\alpha<1$ and inequality (2.4) in \cite{GM86} for $\alpha>1$.

\begin{lemma}
\label{lem-2-1}
For every $\alpha>0$, there is a constant $\gamma$ depending only on $\alpha$ such that
\begin{align*}
\frac{1}{\gamma}\left|| b|^{\alpha-1} b-|a|^{\alpha-1} a\right|\leq(|a|+|b|)^{\alpha-1}
|b-a|\leq \gamma\left||b|^{\alpha-1} b-|a|^{\alpha-1} a \right|
\end{align*}
for all $a,b\in\mathbb{R}$.
\end{lemma}
To work with the term involving the time derivative, we shall use auxiliary functions $\bm{h}_\pm$ defined as
\begin{align}
\label{2.1}
\bm{h}_\pm(w,k):=\pm q \int_k^w|s|^{q-1}(s-k)_{\pm}\,ds,
\end{align}
for $k,w\in\mathbb{R}$ and $q>0$. It is easy to check that $\bm{h}_\pm(w,k)\geq 0$. We also write
\begin{align*}
\bm{h}(w,k):= q \int_k^w|s|^{q-1}(s-k)\,ds.
\end{align*}

The next lemma can be deduced with the help of Lemma \ref{lem-2-1}.

\begin{lemma} [Lemma 2.2, \cite{BDL21}]
\label{lem-2-2}
Let $q>0$. There exists a constant $\gamma=\gamma(q)>0$ such that for all $a,b\in\mathbb{R}$, there holds
\begin{align*}
\frac{1}{\gamma}(|a|+|b|)^{q-1}|a-b|^2 \leq \bm{h}(a,b) \leq \gamma(|a|+|b|)^{q-1}|a-b|^2
\end{align*}
and
\begin{align*}
\frac{1}{\gamma}(|a|+|b|)^{q-1}(a-b)_{ \pm}^2 \leq \bm{h}_{ \pm}(a, b) \leq \gamma(|a|+|b|)^{q-1}(a-b)_{ \pm}^2.
\end{align*}
\end{lemma}

What follows is a necessary tool for dealing with the nonlocal term to discuss the integral-type Harnack inequality. 
\begin{lemma} [Lemma 2.9, \cite{BGK22}]
\label{lem-2-3}
Let $a, b>0, \tau_1, \tau_2 \geq 0$ and $p>1$. Then there exists a constant $\gamma =\gamma(p)>1$ such that
\begin{align*}
|b-a|^{p-2}(b-a)\left(\tau_1^p a^{-\varepsilon}-\tau_2^p b^{-\varepsilon}\right) \geq & \gamma \xi(\varepsilon)\left|\tau_2 b^{\frac{\alpha}{p}}-\tau_1 a^{\frac{\alpha}{p}}\right|^p \\
& -\left(\xi(\varepsilon)+1+\varepsilon^{-(p-1)}\right)|\tau_2-\tau_1|^p\left(b^\alpha+a^\alpha\right),
\end{align*}
where $\varepsilon \in(0, p-1)$, $\alpha:=p-1-\varepsilon$, $\xi(\varepsilon)=\frac{\varepsilon p^p}{\alpha}$ if $0<\alpha<1$, and $\xi(\varepsilon)=\varepsilon(\frac{p}{\alpha})^p$ otherwise.
\end{lemma}

We give a fast geometric convergence lemma from \cite[Lemma 4.1]{D93}.

\begin{lemma} 
\label{lem-2-4}
Let $\{Y_j\}_{j=0}^\infty$ be a sequence of positive numbers satisfying
\begin{align*}
Y_{j+1}\leq Kb^jY_j^{1+\delta},\quad j=0,1,2, \ldots
\end{align*}
for some constants $K$, $b>1$ and $\delta>0$. If
\begin{align*}
Y_0\leq K^{-\frac{1}{\delta}}b^{-\frac{1}{\delta^2}},
\end{align*}
then we have $Y_j\rightarrow 0$ as $j\rightarrow 0$.
\end{lemma}

The following iteration lemma is displayed in \cite[Lemma 6.1]{G03}.

\begin{lemma}
\label{lem-2-5}
Let constants $A, B, C\geq 0$, $\alpha>\beta\geq0$ and $\theta\in (0, 1)$. Suppose that $f: [r, \rho] \rightarrow [0, \infty)$ is a bounded function that satisfies 
\begin{align*}
f(R_1) \leq \theta f(R_2)+\frac{A}{(R_2-R_1)^\alpha}+\frac{B}{(R_2-R_1)^\beta}+C 
\end{align*}
for all $r<R_1<R_2<\rho$. Then we have
\begin{align*}
f(r)\leq \gamma(\alpha, \theta)\left[\frac{A}{(\rho-r)^\alpha}+\frac{B}{(\rho-r)^\beta}+C\right].
\end{align*}
\end{lemma} 

We now present a Poincar\'{e}-type inequality from\cite [Chapter \uppercase\expandafter{\romannumeral1}, Proposition 2.1]{D93}. 

\begin{lemma}
\label{lem-2-6}
Suppose that $\Omega\subset\mathbb{R}^N$ is a bounded convex set. Let $\varphi \in C(\overline{\Omega})$ satisfy $0 \leq \varphi \leq 1$, and the sets $[\varphi>k]$ are convex for any $k\in(0,1)$. Let $v\in W^{1, p}(\Omega)$, and the set 
$$\mathcal{E}:=[v=0] \cap[\varphi=1]$$
has positive measure. Then we have
\begin{align*}
\left(\int_{\Omega} \varphi|v|^p \,dx\right)^{\frac{1}{p}} \leq \gamma \frac{(\operatorname{diam} \Omega)^N}{|\mathcal{E}|^{\frac{N-1}{N}}}\left(\int_\Omega \varphi|\nabla v|^p \,dx\right)^{\frac{1}{p}},
\end{align*}
where $\gamma>0$ depends only on $N$ and $p$, but independent of $v$ and $\varphi$.
\end{lemma}

We finally state a parabolic Sobolev embedding lemma.

\begin{lemma} [Chapter \uppercase\expandafter{\romannumeral1}, Proposition 3.1, \cite{D93}]
\label{lem-2-7}
Suppose that $E\subset\mathbb{R}^N$ is a bounded domain. Let $m, p>1$ and $q=p\frac{N+m}{N}$. Then for every 
$$u\in L^{\infty}(0, T; L^m(E)) \cap L^p(0, T; W_0^{1, p}(E)),$$
we have
\begin{align*}
\iint_{E_T}|u|^q \,dxdt \leq \gamma \iint_{E_T}|\nabla u|^p \,dxdt\left(\operatorname*{ess \sup}_{t\in[0,T]}  \int_{E\times\{t\}}|u|^m\,dx\right)^{\frac{p}{N}},
\end{align*}
where $\gamma>0$ only depends on $N,p$ and $q$. 
\end{lemma}

\subsection{Time mollification}

The exponential mollification in time will be introduced in this subsection to solve the difficulty that weak solutions do not have a time derivation generally. This technique is extracted from \cite{KL06}. For any $v\in L^1(E_T)$ and $h>0$, define
\begin{align*}
[\![v]\!]_h(x,t):=\frac{1}{h} \int_0^t e^{\frac{s-t}{h}} v(x,s)\,ds
\end{align*}
and
\begin{align*}
[\![v]\!]_{\bar{h}}(x, t):=\frac{1}{h} \int_t^T e^{\frac{t-s}{h}} v(x, s)\,ds.
\end{align*}	

The following fundamental properties of mollified functions are given in \cite [Appendix]{BDM13}.

\begin{lemma}
\label{lem-2-8}
Let $r\geq 1$. Then the following conclusions hold. 

\begin{itemize}
\item [(i)] If $v \in L^r(E_T)$, then $[\![v]\!]_h \in L^r (E_T)$ and $\|[\![v]\!]_h\|_{L^r(E_T)}\leq\|v\|_{L^r(E_T)}$. In addition, we have
$[\![v]\!]_h \rightarrow v$ strongly in $L^r(E_T)$ and almost everywhere on $E_T$ as $h \rightarrow 0$. 

\item [(ii)] There almost everywhere on $E_T$ holds
\begin{align*}
\partial_t[\![v]\!]_h=\frac{1}{h}\left(v-[\![v]\!]_h\right),\quad \partial_t[\![v]\!]_{\bar{h}}=\frac{1}{h}\left([\![v]\!]_{\bar{h}}-v\right).
\end{align*}

\item [(iii)] If $v\in L^r(E_T)$, then $[\![v]\!]_h$ and $[\![v]\!]_{\bar{h}}$ belong to $C\left([0, T] ; L^r(E)\right)$.

\item [(iv)] If $Du\in L^r(E_T)$, then $D [\![v]\!]_h=[\![Dv]\!]_h \rightarrow Dv$ strongly in $L^r(E_T)$ and almost everywhere on $E_T$ as $h \rightarrow 0$. 

\item [(v)] If $v\in C([0,T];L^r(E))$, then $[\![v]\!]_h(\cdot,t) \rightarrow v(\cdot,t)$ strongly in $L^r(E)$ and almost everywhere on $E$ for any $t\in(0,T]$ as $h\rightarrow 0$.
\end{itemize}

The same statements of $(\rm{{\romannumeral 1}}), (\rm{{\romannumeral 4}})$ and $(\rm{{\romannumeral 5}})$  also apply to $[\![v]\!]_{\bar{h}}$.
\end{lemma}

\section{Comparison principles}
\label{sec3}
In this part, we will consider the continuity with respect to the time variable of weak solutions and establish a comparison principle. We introduce the Cauchy-Dirichlet problem
\begin{align}
\label{3.1}
\left\{\begin{array}{cl}
\partial_t(|u|^{q-1} u)-\mathrm{div} (|\nabla u|^{p-2}\nabla u)+\mathcal{L} u=0 & \text { in } E_T, \\
u=0 & \text { in } \mathbb{R}^N\backslash E \times(0,T], \\
u(\cdot, 0)=u_0 & \text { in } E,
\end{array}\right.
\end{align}
where $p>1, q>0$, $\mathcal{L}$ is defined as in \eqref{1.2}, and $u_0\in L^{q+1}(E)$. 

We give the definition of weak solutions to \eqref{3.1} as below.

\begin{definition}
We say a function
\begin{align*}
u\in C([0,T];L^{q+1}) \cap L^p(0,T;W_0^{1,p}(E))\cap L^\infty(0,T;L_{sp}^{p-1}(\mathbb{R}^N))
\end{align*}
is a weak solution to \eqref{3.1}, if there holds
\begin{align}
\label{3.2}
\iint_{E_T}(|u_0|^{q-1} u_0-|u|^{q-1}u)\partial_t \zeta+|\nabla u|^{p-2}\nabla u \cdot \nabla \zeta\,dxdt+\int_{0}^{T} \mathcal{E}(u,\zeta,t)\,dt =0
\end{align}
for any test function $\zeta \in W^{1,q+1}(0,T;L^{q+1}(E)) \cap L^p(0,T;W_0^{1, p}(E))$ with $\zeta(\cdot,T)=0$, where
\begin{align*}
\mathcal{E}(u,\zeta,t):=\int_{\mathbb{R}^N}\int_{\mathbb{R}^N} K(x,y,t)|u(x,t)-u(y,t)|^{p-2}(u(x,t)-u(y,t))(\zeta(x,t)-\zeta(y,t))\,dydx.
\end{align*}
\end{definition}

\subsection{Parabolicity}

In what follows, we show Eq. \eqref{1.1} is parabolicity, which describes the property that if $u$ is a weak subsolution (super-) to \eqref{1.1}, then the corresponding truncation function is also a weak subsolution (super-) to \eqref{1.1}.

\begin{proposition}
\label{pro-3-2}
Let $p>1$ and $q>0$. If $u$ is a weak subsolution (super-) to \eqref{1.1}, then the truncation function $u_k=k+(u-k)_+$, $u_k=k-(u-k)_-$ with $k\in\mathbb{R}$ is a weak subsolution (super-) to \eqref{1.1}. 
\end{proposition}

\begin{proof}
We verify this result holds for subsolutions, while the claim for supersolutions can proceed in the same way. For $\sigma>0$, test the weak formulation \eqref{1.7} with function
\begin{align*}
\varphi_h:=\frac{\zeta([\![u]\!]_{\bar{h}}-k)_+}{([\![u]\!]_{\bar{h}}-k)_++\sigma},
\end{align*}	
where $\zeta\in W_0^{1,q+1}(0,T;L^{q+1}(E)) \cap L^p(0, T; W_0^{1, p}(E).$ The term including the time derivative, and the local term can be dealt with the same as in  \cite[Proposition 4.7]{BDG23}, we only concentrate on the nonlocal term. Letting $h\rightarrow 0$ and using Lemma \ref{lem-2-8} $(\rm{{\romannumeral 1}})$, we have
\begin{align*}
& \lim _{h \rightarrow 0} \int_{0}^{T}\int_{\mathbb{R}^N} \int_{\mathbb{R}^N} U(x,y,t)(\varphi_h(x,t)-\varphi_h(y,t))\,d\mu dt  \\
& = \int_{0}^{T}\int_{\mathbb{R}^N} \int_{\mathbb{R}^N} U(x,y,t)\left[\frac{\zeta(u-k)_+(x,t)}{(u-k)_+(x,t)+\sigma}-\frac{\zeta(u-k)_+(y,t)}{(u-k)_+(y,t)+\sigma}\right]\,d\mu dt.
\end{align*}
We shall send $\sigma\rightarrow 0$, and get
\begin{align*}
&-\iint_{E_T}  |u_k|^{q-1}u_k\partial_t \zeta\,dxdt+\iint_{E_T} |\nabla u|^{p-2}\nabla u\cdot\nabla\zeta\chi_{[u>k]}\,dxdt\\
&\ \ \ +\int_{0}^{T}\int_{\mathbb{R}^N} \int_{\mathbb{R}^N}U(x,y,t)\left[\zeta(x,t)\chi_{[u(x,t)>k]}-\zeta(y,t)\chi_{[u(y,t)>k]}\right]\,d\mu dt\leq 0.
\end{align*}
Moreover, one can check that
\begin{align*}
&\int_{0}^{T}\int_{\mathbb{R}^N} \int_{\mathbb{R}^N}|u_k(x,t)-u_k(y,t)|^{p-2}(u_k(x,t)-u_k(y,t))(\zeta(x,t)-\zeta(y,t))\,d\mu dt\\
\leq&\int_{0}^{T}\int_{\mathbb{R}^N} \int_{\mathbb{R}^N}U(x,y,t)\left[\zeta(x,t)\chi_{[u(x,t)>k]}-\zeta(y,t)\chi_{[u(y,t)>k]}\right]\,d\mu dt.
\end{align*}
Thus, we conclude that
\begin{align*}
&-\iint_{E_T} |u_k|^{q-1}u_k\partial_t \zeta\,dxdt+\iint_{E_T} |\nabla u_k|^{p-2}\nabla u_k\cdot\nabla\zeta\,dxdt\\
&\ \ \ +\int_{0}^{T}\int_{\mathbb{R}^N} \int_{\mathbb{R}^N}|u_k(x,t)-u_k(y,t)|^{p-2}(u_k(x,t)-u_k(y,t))(\zeta(x,t)-\zeta(y,t))\,d\mu dt\leq 0
\end{align*}
for every nonnegative $\zeta \in W^{1,q+1}_0(0,T;L^{q+1}(E)) \cap L^p(0,T;W_0^{1, p}(E))$.
\end{proof}

\subsection{Time continuity of weak solutions}

The forthcoming proposition tells the weak solutions of \eqref{1.1} belong to $C\left([0, T]; L^{q+1}(E)\right)$. We refer the  readers to \cite{N22} for details.
\begin{proposition}
\label{pro-3-3}
 Let $p>1$, $q>0$ and let $u\in L^\infty(0, T; L^{q+1}(E))\cap L^p(0, T; W_0^{1, p}(E))\cap L^\infty(0, T; L_{sp}^{p-1}(\mathbb{R}^N))$ satisfy the weak formula \eqref{3.2} for some $u_0\in L^{q+1}(E)$. Then there is a representative  $u\in C\left([0, T]; L^{q+1}(E)\right)$, indicating that
\begin{align*}
\lim_{t\rightarrow 0} \int_E |u(\cdot,t) -u_0|^{q+1}\,dx=0.
\end{align*}
\end{proposition}

\subsection {Comparison principle} 

Let us denote the Lipschitz function  $H_\delta(s)$ by
\begin{align*}
H_\delta(s):=\left\{\begin{array}{cl}
1, & \text { for } s \geq \delta, \\
\frac{1}{\delta} s, & \text { for } 0<s<\delta, \\
0, & \text { for } s \leq 0 .
\end{array}\right.
\end{align*}
Define the functions
\begin{align*}
h_\delta(z,z_0):=\int_{z_0}^z H_\delta(s-z_0) q s^{q-1}\,ds\ \ \text { for } z, z_0 \in \mathbb{R}_{\geq 0},
\end{align*}
\begin{align*}
\widehat{h}_\delta(z, z_0):=\int_{z_0}^z \widehat{H}_\delta(s-z_0) q s^{q-1}\,ds \ \  \text { for } z, z_0 \in \mathbb{R}_{\geq 0},
\end{align*}
where $\widehat{H}_\delta(s):=-H_\delta(-s)$ is the odd reflection of $H_\delta$.
In the rest of this section, we write
\begin{align*}
&V(x,y,t)=|v(x,t)-v(y,t)|^{p-2}(v(x,t)-v(y,t)),\\ &W(x,y,t)=|w(x,t)-w(y,t)|^{p-2}(w(x,t)-w(y,t)).
\end{align*}

Now, we provide a comparison principle for nonnegative weak solutions to \eqref{1.1} and \eqref{3.1}. The comparison of two weak solutions in $\mathbb{R}^N\backslash E\times(0, T)$ can be directly checked through the boundary value of \eqref{3.1} and the non-negativity of weak solutions to \eqref{1.1}. Hence, we just need to give the assumption on the initial value to develop the comparison principle on cylinders.

\begin{proposition}
\label{pro-3-4}
Let $p>1, q>0$, and let $w$ be a nonnegative weak solution to $\eqref{1.1}$ and $v$ be a nonnegative weak solution to \eqref{3.1} with $v_0 \geq 0$. If $v_0 \leq w(\cdot, 0)$ a.e. in $E$, then $v \leq w$ a.e. in $E_T$. 
\end{proposition}

For proving Proposition \ref{pro-3-4}, we discuss the result for the function $v$ first.
\begin{lemma}
\label{lem-3-5}
Let $p>1$, $q>0$, and let $v$ be a nonnegative weak solution to \eqref{3.1} with $v_0 \geq 0$. Suppose that $\widetilde{v} \in L^{q+1}(E) \cap W_0^{1, p}(E)$ is a nonnegative function. For any $\psi \in C_0^{\infty}(\mathbb{R}^N\times(0,T) )$, there holds that
\begin{align*}
&\iint_{E_T}-h_\delta(v,\widetilde{v}) \partial_t \psi+|\nabla v|^{p-2}\nabla v\cdot \nabla\left[H_\delta(v-\widetilde{v}) \psi\right]\,dxdt\\
 &+\int_{0}^{T}\int_{\mathbb{R}^N}\int_{\mathbb{R}^N} V(x,y,t) [H_\delta(v(x,t)-\widetilde{v}(x))\psi (x,t)-H_\delta(v(y,t)-\widetilde{v}(y))\psi (y,t)]\,d\mu dt=0.
\end{align*}
\end{lemma}
\begin{proof}
According to $v=0$ in $\mathbb{R}^N\backslash E\times(0,T)$ along with $\widetilde{v}=0$ in $\mathbb{R}^N\backslash E$, we get
\begin{align*}
H_\delta(v(\cdot, t)-\widetilde{v}(\cdot))=0 \quad \text { in } \mathbb{R}^N\backslash E
\end{align*}
for a.e. $t\in(0,T)$, which implies $\zeta=H_\delta\left([\![v]\!]_h-\widetilde{v}\right) \psi$ is a admissible test function in \eqref{3.2}. Due to $\zeta(\cdot,0)=0$, the term including $v_0$ vanishes.
In light of Lemma \ref{lem-2-8} $(\rm{{\romannumeral 1}})$ and $(\rm{{\romannumeral 4}})$, we have 
\begin{align*}
H_\delta([\![v]\!]_h-\widetilde{v}) \psi \rightarrow H_\delta(v-\widetilde{v})\psi  \quad \text{as } h\rightarrow 0
\end{align*}
and
\begin{align*}
\nabla[H_\delta([\![v]\!]_h-\widetilde{v}) \psi] \rightarrow \nabla[H_\delta(v-\widetilde{v})\psi]  \quad \text{as } h\rightarrow 0.
\end{align*}
Then we obtain
\begin{align*}
\lim_{h \rightarrow 0} \iint_{E_T} |\nabla v|^{p-2}\nabla v\cdot\nabla\zeta\,dxdt
=\iint_{E_T} |\nabla v|^{p-2}\nabla v\cdot \nabla[H_\delta(v-\widetilde{v})\psi]\,dxdt
\end{align*}
and 
\begin{align*}
&\lim_{h \rightarrow 0}\int_{0}^{T}\int_{\mathbb{R}^N}\int_{\mathbb{R}^N}V(x,y,t)(\zeta(x,t)-\zeta(y,t))\,d\mu dt\\
=&\int_{0}^{T}\int_{\mathbb{R}^N}\int_{\mathbb{R}^N}V(x,y,t)[H_\delta(v(x,t)-\widetilde{v}(x))\psi (x,t)-H_\delta(v(y,t)-\widetilde{v}(y))\psi(y,t)]\,d\mu dt.
\end{align*}
Thanks to Lemma \ref{lem-2-8} $(\rm{{\romannumeral 2}})$, we estimate the time part as 
\begin{align}
\label{3.3}
\iint_{E_T} & -v^q \partial_t \zeta\,dxdt=\iint_{E_T}(-[\![v]\!]_h^q+[\![v]\!]_h^q-v^q) \partial_t \zeta\,dxdt\nonumber\\
=&-\iint_{E_T} [\![v]\!]_h^q \partial_t \zeta\,dxdt+\iint_{E_T}([\![v]\!]_h^q-v^q) H_\delta([\![v]\!]_h-\widetilde{v}) \partial_t \psi\,dxdt\nonumber\\
&+\iint_{E_T}([\![v]\!]_h^q-v^q) H_\delta^{\prime}([\![v]\!]_h-\widetilde{v}) \frac{1}{h}(v-[\![v]\!]_h) \psi\,dxdt\nonumber\\
\leq & \iint_{E_T} \partial_t [\![v]\!]_h^q \zeta\,dxdt +\iint_{E_T}\left([\![v]\!]_h^q-v^q\right) H_\delta([\![v]\!]_h-\widetilde{v}) \partial_t \psi\,dxdt\nonumber\\
=&-\iint_{E_T} h_\delta([\![v]\!]_h, \widetilde{v}) \partial_t \psi\,dxdt +\iint_{E_T}([\![v]\!]_h^q-v^q) H_\delta([\![v]\!]_h-\widetilde{v}) \partial_t \psi\,dxdt.
\end{align}
Observe that the second term on the right-hand side of \eqref{3.3} vanishes as $h\rightarrow 0$ by Lemma \ref{lem-2-8} $(\rm{{\romannumeral 1}})$. In addition, we can find
\begin{align*}
\left|h_\delta([\![v]\!]_h, \widetilde{v})-h_\delta(v, \widetilde{v})\right|=\left|\int_v^{[\![v]\!]_h} H_\delta(s-\widetilde{v}) q s^{q-1}\,ds\right|\leq\left|[\![v]\!]_h^q-v^q\right|,
\end{align*}
which combines with Lemma \ref{lem-2-8} $(\rm{{\romannumeral 1}})$ leads to
\begin{align*}
-\lim_{h \rightarrow 0} \iint_{E_T} h_\delta([\![v]\!]_h, \widetilde{v}) \partial_t \psi\,dxdt =-\iint_{E_T} h_\delta(v, \widetilde{v}) \partial_t \psi\,dxdt .
\end{align*}
So far, we get the desired result with ``$\geq$ " by collecting the above estimates. Analogously, we can prove the reverse inequality by choosing $\zeta=H_\delta([\![v]\!]_{\bar{h}}-\widetilde{v}) \psi$ as a test function.
\end{proof}

Indeed, if taking $\zeta=\widehat{H}_\delta([\![w]\!]_h-\widetilde{w}) \psi$ as a test function in \eqref{3.2} and running similarly, we can obtain the result for the function $w$.

\begin{lemma}
\label{lem-3-6}
Let $p>1, q>0$, and let $w$ be a nonnegative weak solution to $\eqref{1.1}$. Suppose that $\widetilde{w} \in L^{q+1}(E) \cap W_0^{1, p}(E)$ is a nonnegative function. For any $\psi \in C_0^{\infty}\left(\mathbb{R}^N\times(0,T) \right)$, there holds that
\begin{align*}
&\iint_{E_T}-\widehat{h}_\delta(w, \widetilde{w}) \partial_t \psi+|\nabla w|^{p-2}\nabla w\cdot\nabla\left(\widehat{H}_\delta(w-\widetilde{w}) \psi\right)\,dxdt\\
&+\int_{0}^{T}\int_{\mathbb{R}^N}\int_{\mathbb{R}^N} W(x,y,t) \left(\widehat{H}_\delta(w(x,t)-\widetilde{w}(x)) \psi(x,t)-\widehat{H}_\delta(w(y,t)-\widetilde{w}(y)) \psi(y,t)\right)\,d\mu dt=0.
\end{align*}
\end{lemma}

Now we are ready to prove Proposition \ref{pro-3-4} by using the argument proposed in \cite{O96}.

\begin{proof}[\textbf{Proof of Proposition \ref{pro-3-4}}]
We set
\begin{align*}
(x,t_1,t_2) \in \widetilde{Q}:=E \times(0,T)^2.
\end{align*}
Let $0\leq\psi \in C_0^{\infty}((0,T)^2)$, and extend functions $v$ and $w$ to $\widetilde{Q}$ with
\begin{align*}
v(x,t_1,t_2):=v(x,t_1), \quad w(x,t_1,t_2):=w(x,t_2).
\end{align*}
For fixed $\delta>0$ and a.e. $t_2 \in(0, T)$, we note that $H_\delta$, $\widetilde{v}(x)=w(x,t_2)=:w_{t_2}(x)$, $\widetilde{v}(y)=w(y,t_2)=:w_{t_2}(y)$, and $\psi_{t_2}(t_1):=\psi(t_1, t_2)$ for $t_1\in(0,T)$ are allowed in Lemma \ref{lem-3-5} based on the fact $v=0$ in $\mathbb{R}^N\backslash E\times(0,T]$. Thus, it holds that
\begin{align}
\label{3.4}
&\iint_{E_T}-h_\delta(v,w_{t_2}) \partial_{t_1} \psi_{t_2}+
|\nabla v|^{p-2}\nabla v\cdot \nabla\left[H_\delta(v(x,t_1)-w_{t_2}) \psi_{t_2}\right]\,dxdt_1\nonumber\\
&+\int_{0}^{T}\int_{\mathbb{R}^N}\int_{\mathbb{R}^N} V(x,y,t_1) [H_\delta(v(x,t_1)-w_{t_2}) -H_\delta(v(y,t_1)-w_{t_2}) ]\psi_{t_2}\,d\mu dt_1=0.
\end{align}
Likewise, for fixed $\delta>0$ and a.e. $t_1 \in(0, T)$, we note that $\widehat{\mathcal{H}}_\delta$, $\widetilde{w}(x)=v(x,t_1)=:v_{t_1}(x)$, $\widetilde{w}(y)=v(y,t_1)=:v_{t_1}(y)$, and $\psi_{t_1}(t_2):=\psi(t_1,t_2)$ for $t_2\in(0,T)$ are permitted in Lemma \ref{lem-3-6} due to $v=0$ in $\mathbb{R}^N\backslash E\times(0,T]$. Subsequently, we get
\begin{align}
\label{3.5}
&\iint_{E_T}-\widehat{h}_\delta(w,v_{t_1}) \partial_{t_2} \psi_{t_1}+
|\nabla w|^{p-2}\nabla w\cdot\nabla\left(\widehat{H}_\delta(w(x,t_2)-v_{t_1}) \psi_{t_1}\right)\,dxdt_2\nonumber\\
&+\int_{0}^{T}\int_{\mathbb{R}^N}\int_{\mathbb{R}^N} W(x,y,t_2) \left(\widehat{H}_\delta(w(x,t_2)-v_{t_1}) -\widehat{H}_\delta(w(y,t_2)-v_{t_1}) \right)\psi_{t_1}\,d\mu dt_2=0.
\end{align}
Integrating \eqref{3.4} over $t_2\in(0,T)$, \eqref{3.5} over $t_1\in(0,T)$, and adding two integral equalities yields that
\begin{align}
\label{3.6}
&\iiint_{\widetilde{Q}} -\left(h_\delta(v,w) \partial_{t_1} \psi+\widehat{h}_\delta(w,v) \partial_{t_2} \psi\right)\,dxdt_1dt_2 \nonumber\\
&+\iiint_{\widetilde{Q}}\left(|\nabla v|^{p-2}\nabla v-|\nabla w|^{p-2}\nabla w\right)\cdot \nabla\left[H_\delta(v-w) \right]\psi\,dxdt_1dt_2\nonumber\\
&+\int_{0}^{T}\!\int_{0}^{T}\!\int_{\mathbb{R}^N}\!\int_{\mathbb{R}^N}\left[V(x,y,t_1,t_2)-W(x,y,t_1,t_2)\right]\nonumber\\
&\ \ \ \ \ \ \ \ \ \ \ \ \ \ \ \ \ \ \ \ \ \ \ \  \times \left[H_\delta(v-w)(x,t_1,t_2)-H_\delta(v-w)(y,t_1,t_2)\right]\psi\,d\mu dt_1dt_2=0,
\end{align}
where we employed the property $H_\delta(z)=-\widehat{H}_\delta(-z)$.
Before sending $\delta\rightarrow 0$, we treat the local and nonlocal terms that do not possess a regular limit.
For the second term in \eqref{3.6}, recalling the algebraic inequality
\begin{align}
\label{3.7}
\left(|\eta|^{p-2}\eta-|\zeta|^{p-2}\zeta\right)(\eta-\zeta)\geq 0
\end{align}
holds for all $\eta,\zeta\in\mathbb{R}$, thus we have
\begin{align*}
&(|\nabla v|^{p-2}\nabla v-|\nabla w|^{p-2}\nabla w)\cdot\nabla\left[H_\delta(v-w)\right]\\
=&H_\delta^\prime(v-w)(\nabla v-\nabla w)(|\nabla v|^{p-2}\nabla v-|\nabla w|^{p-2}\nabla w)\geq 0.
\end{align*}
For the nonlocal term, it infers from the Mean Value Theorem that there exists  a function $\xi$ which lies between $(v-w)(x,t_1,t_2)$ and $(v-w)(y,t_1,t_2)$ such that
\begin{align*}
&\left(V(x,y,t_1,t_2)-W(x,y,t_1,t_2)\right)\left[H_\delta(v-w)(x,t_1,t_2)-H_\delta(v-w)(y,t_1,t_2)\right]\\
=&[|v(x,t_1,t_2)-v(y,t_1,t_2)|^{p-2}(v(x,t_1,t_2)-v(y,t_1,t_2))\\
&-|w(x,t_1,t_2)-w(y,t_1,t_2)|^{p-2}(w(x,t_1,t_2)-w(y,t_1,t_2))]\\
&\times H^\prime_\delta(\xi)[(v(x,t_1,t_2)-v(y,t_1,t_2))-(w(x,t_1,t_2)-w(y,t_1,t_2))]\geq 0,
\end{align*}
where we used \eqref{3.7} again. Dropping the nonnegative terms in \eqref{3.6} and letting $\delta\rightarrow 0$, we derive
\begin{align}
\label{3.8}
\limsup _{\delta \rightarrow 0} \iiint_{\widetilde{Q}}-\left(h_\delta(v, w) \partial_{t_1} \psi+\widehat{h}_\delta(w, v) \partial_{t_2} \psi\right)\,dxdt_1dt_2 \leq 0.
\end{align}
Once we obtain the integral inequality \eqref{3.8}, Proposition \ref{pro-3-4} follows from the argument in \cite[Proposition 4.16]{BDG23}.
\end{proof}

\section{Energy estimate and local boundedness}
\label{sec4}
\subsection{Energy estimate} In this section, we first present the Caccioppoli-type inequality which can proceed similarly as in \cite{BGK23}.
\begin{lemma}
	\label{lem-4-1}
	Let $p>1$ and $q>0$. Let $u$ be a nonnegative weak subsolution to \eqref{1.1} and let $Q_{\rho,s}=K_\rho(x_0)\times(t_0-s,t_0]\subset\subset E_T$. For any nonnegative, piecewise smooth cutoff function $\psi$ vanishing on $\partial K_\rho(x_0)\times(t_0-s,t_0)$, there exists a constant $\gamma(N,p,s,q,\Lambda)>0$ such that
	\begin{align}
	\label{4.1}
	&\operatorname*{ess \sup}_{t_0-s<t<t_0} \int_{K_\rho(x_0)\times\{t\}}  \bm{h}_+(u,k)\psi^p(x, t)\,dx+\iint_{Q_{\rho,s}}|\nabla(u-k)_+|^p\psi^p(x,t)\,dxdt\nonumber\\
	&+ \int_{t_0-s}^{t_0} \int_{K_\rho(x_0)} \int_{K_\rho(x_0)}\left|(u-k)_+(x,t)\psi(x, t)-(u-k)_+(y,t)\psi(y,t)\right|^p\, d\mu dt \nonumber\\
	\leq& \gamma \iint_{Q_{\rho,s}}\bm{h}_+(u,k)|\partial_t\psi|+(u-k)_+^p|\nabla\psi|^p\,dxdt\nonumber\\
	&+\gamma\int_{t_0-s}^{t_0} \int_{K_\rho(x_0)} \int_{K_\rho(x_0)} \max \{(u-k)_+(x, t),(u-k)_+(y,t)\}^p|\psi(x,t)-\psi(y,t)|^p\,d\mu dt \nonumber\\
	&+\gamma \mathop{\mathrm{ess}\,\sup}_{\stackrel{t_0-s<t<t_0}{x\in \mathrm{supp}\,\psi(\cdot,t)}}\int_{\mathbb{R}^N \backslash K_\rho(x_0)} \frac{(u-k)_+^{p-1}(y,t)}{|x-y|^{N+sp}}\,dy\iint_{Q_{\rho,s}}(u-k)_+\psi^p(x,t)\,dxdt\nonumber\\
	&+\int_{K_\rho(x_0)\times\{t_0-s\}}  \bm{h}_+(u,k)\psi^p(x, t)\,dx,
	\end{align}
	where $k\in \mathbb{R}$ and the function $\bm{h}_+$ is defined as in \eqref{2.1}.
\end{lemma}
\subsection{Local boundedness}
In the following, we devote to showing the quantitative $L^\infty$ bound of weak solutions to \eqref{1.1} presented in Theorem \ref{thm-1-2}. To start with, we deduce an iterative inequality.
\subsubsection{An iterative inequality for $r\geq q+1$}
Let $(x_0,t_0)=(0,0)$. For $\sigma\in(0,1)$, define decreasing sequences
\begin{align*}
\rho_0=\rho,\ \ \rho_j=\sigma\rho+2^{-j}(1-\sigma)\rho,\ \ \widetilde{\rho}_j=\frac{\rho_j+\rho_{j+1}}{2}, \ \ \widehat{\rho}_j=\frac{3\rho_j+\rho_{j+1}}{4}, \ \ j=0,1,2\ldots,
\end{align*}
\begin{align*}
s_0=s,\ \ s_j=\sigma s+2^{-j}(1-\sigma)s,\ \ \widetilde{s}_j=\frac{s_j+s_{j+1}}{2}, \ \ \widehat{s}_j=\frac{3s_j+s_{j+1}}{4}, \ \ j=0,1,2\ldots.
\end{align*}
Set the domains
\begin{align*}
K_j=K_{\rho_j}, \ \ \widetilde{K}_j=K_{\widetilde{\rho}_j}, \ \ \widehat{K}_j=K_{\widehat{\rho}_j}, \ \ j=0,1,2\ldots,
\end{align*}
\begin{align*}
Q_j=K_j \times(-s_j, 0], \ \  \widetilde{Q}_j=\widetilde{K}_j \times(-\widetilde{s}_j,0], \ \ \widehat{Q}_j=\widehat{K}_j \times (-\widehat{s}_j,0], \ \ j=0,1,2\ldots.
\end{align*}
Take increasing sequences 
\begin{align*}
k_j=k-\frac{k}{2^j},\ \ \widetilde{k}_j=\frac{k_j+k_{j+1}}{2}, \ \ j=0,1,2\ldots
\end{align*}
with level $k>0$ will be specified later. Consider the cutoff function $\zeta\in C^\infty_0(Q_j)$ vanishing outside $\widehat{Q}_j$, satisfying
\begin{align*}
0\leq\zeta\leq 1,\ \ |\nabla \zeta| \leq \frac{2^{j+2}}{(1-\sigma) \rho}, \ \  |\partial_t \zeta| \leq \frac{2^{j+2}}{(1-\sigma) s}, \ \ \zeta\equiv 1 \quad \text{in } \widetilde{Q}_j.
\end{align*}
Applying the energy estimate \eqref{4.1} in this framework yields that
\begin{align}
\label{4.2}
&\operatorname*{ess \sup}_{-\widetilde{s}_j<t<0} \int_{\widetilde{K}_j\times\{t\}}  \bm{h}_+(u,\widetilde{k}_j)\,dx+\iint_{\widetilde{Q}_j}|\nabla(u-\widetilde{k}_j)_+|^p\,dxdt\nonumber\\
\leq& \frac{\gamma 2^j}{(1-\sigma)s} \iint_{Q_j}\bm{h}_+(u,\widetilde{k}_j)\,dxdt+\frac{\gamma 2^{pj}}{(1-\sigma)^p\rho^p}\iint_{Q_j}(u-\widetilde{k}_j)_+^p\,dxdt\nonumber\\
&+\frac{\gamma}{(1-\sigma)^p\rho^p}\int_{-s_j}^{0} \int_{K_j} \int_{K_j}\frac{\max \{(u-\widetilde{k}_j)_+^p(x,t),(u-\widetilde{k}_j)_+^p(y,t)\}}{|x-y|^{N-(1-s)p}} \,dxdydt\nonumber\\
&+\gamma \mathop{\mathrm{ess}\,\sup}_{\stackrel{-s_j<t<0}{x\in \mathrm{supp}\,\zeta(\cdot,t)}}\int_{\mathbb{R}^N \backslash K_j} \frac{(u(y,t)-\widetilde{k}_j)_+^{p-1}}{|x-y|^{N+sp}}\,dy\iint_{Q_j}(u(x,t)-\widetilde{k}_j)_+\zeta^p(x,t)\,dxdt,
\end{align}
where we drop the nonnegative term on the left-hand side, and the constant $\gamma$ depends only on $N,p,s,q,\Lambda$. Next, we estimate the first term on the right-hand side of \eqref{4.2}. On the set $[u>k_{j+1}]$, we have
\begin{align*}
1 \leq \frac{u+\widetilde{k}_j}{u-\widetilde{k}_j} \leq \frac{2 u}{u-\widetilde{k}_j} \leq \frac{2 k_{j+1}}{k_{j+1}-\widetilde{k}_j} \leq 2^{j+3},
\end{align*}
which in conjunction with Lemma \ref{lem-2-2} gives that
\begin{align}
\label{4.3}
\bm{h}_+(u,\widetilde{k}_j) & \geq \frac{1}{\gamma}(u+\widetilde{k}_j)^{q-1}(u-\widetilde{k}_j)_+^2\nonumber \\
 &\geq \frac{1}{\gamma} 2^{-(j+3)(1-q)_+}(u-k_{j+1})_+^{q+1},
\end{align}
where $\gamma>0$ depends only on $q$. Besides, on the set $[u>\widetilde{k}_j]$, 
\begin{align*}
1 \leq \frac{u+\widetilde{k}_j}{u-k_j} \leq \frac{2 u}{u-k_j} \leq \frac{2 \widetilde{k}_j}{\widetilde{k}_j-k_j} \leq 2^{j+3}.
\end{align*}
Therefore, we have by Lemma \ref{lem-2-2} that
\begin{align}
\label{4.4}
\bm{h}_+(u, \widetilde{k}_j) &\leq \gamma(u+\widetilde{k}_j)^{q-1}(u-\widetilde{k}_j)_+^2\nonumber \\
 &\leq \gamma 2^{(j+3)(q-1)_+}(u-k_j)_+^{q+1} \chi_{[u>\widetilde{k}_j]}
\end{align}
with $\gamma$ only depending on $q$. For the last term of \eqref{4.2}, it is easy to find 
\begin{align*}
\frac{|y|}{|y-x|}\leq 1+\frac{|x|}{|y-x|}\leq 1+\frac{\widehat{\rho}_j}{\rho_j-\widehat{\rho}_j}\leq\frac{ 2^{j+4}}{1-\sigma}
\end{align*}
for every $|x|\leq\widehat{\rho}_j \text{ and } |y|\geq \rho_j$. Thus we obtain
\begin{align}
\label{4.5}
&\mathop{\mathrm{ess}\,\sup}_{\stackrel{-s_j<t<0}{x\in \mathrm{supp}\,\zeta(\cdot,t)}}\int_{\mathbb{R}^N \backslash K_j} \frac{(u(y,t)-\widetilde{k}_j)_+^{p-1}}{|x-y|^{N+sp}}\,dy\iint_{Q_j}(u(x,t)-\widetilde{k}_j)_+\zeta^p(x,t)\,dxdt\nonumber\\
\leq& \frac{\gamma 2^{(N+sp)j}}{(1-\sigma)^{N+sp}}\operatorname*{ess \sup}_{-s_j<t<0}\int_{\mathbb{R}^N \backslash K_{\sigma\rho}}\frac{(u(y,t)-k_0)_+^{p-1}}{|y|^{N+sp}}\,dy\iint_{Q_j}(u(x,t)-\widetilde{k}_j)_+\,dxdt\nonumber\\
\leq & \frac{\gamma 2^{(N+sp)j}}{(1-\sigma)^{N+sp}(\sigma\rho)^p}[\mathrm{Tail}_\infty(u;0,\sigma\rho;-s_j,0)]^{p-1}\iint_{Q_j}(u(x,t)-\widetilde{k}_j)_+\,dxdt.
\end{align}
Substituting \eqref{4.3}--\eqref{4.5} into \eqref{4.2}, we arrive at
\begin{align}
\label{4.6}
& 2^{-j(1-q)_+} \operatorname*{ess \sup}_{-\widetilde{s}_j<t<0} \int_{\widetilde{K}_j \times\{t\}}(u-k_{j+1})_+^{q+1}\,dx+\iint_{\widetilde{Q}_j}|\nabla(u-\widetilde{k}_j)_+|^p\,dxdt \nonumber\\
\leq & \frac{\gamma 2^{j[1+(q-1)_+]}}{(1-\sigma) s} \iint_{\widetilde{A}_j}(u-k_j)_+^{q+1}\,dxdt +\frac{\gamma 2^{pj}}{(1-\sigma)^p \rho^p} \iint_{Q_j}(u-\widetilde{k}_j)_+^p\,dxdt\nonumber\\
&+ \frac{\gamma 2^{(N+sp)j}}{(1-\sigma)^{N+sp}(\sigma\rho)^p}[\mathrm{Tail}_\infty(u;0,\sigma\rho;-s_j,0)]^{p-1}\iint_{Q_j}(u-\widetilde{k}_j)_+\,dxdt,
\end{align}
where $\widetilde{A}_j:=[u>\widetilde{k}_j]\cap Q_j$, and $\gamma$ depends only on $N,p,s,q,\Lambda$. 
Note that 
\begin{align}
\label{4.7}
|\widetilde{A}_j|=|[u>\widetilde{k}_j]\cap Q_j|\leq \frac{2^{(j+2)r}}{k^r} \iint_{Q_j}(u-k_j)_+^r\,dxdt,
\end{align}
along with the assumption $r\geq q+1\geq p$ allows us to apply H\"{o}lder's inequality to get
\begin{align}
\label{4.8}
\iint_{\widetilde{A}_j}(u-k_j)_+^{q+1}\,dxdt& \leq\left(\iint_{Q_j}(u-k_j)_+^r\,dxdt \right)^{\frac{q+1}{r}}|\widetilde{A}_j|^{1-\frac{q+1}{r}} \nonumber\\
& \leq  \frac{\gamma 2^{(r-q-1)j}}{k^{r-q-1}} \iint_{Q_j}(u-k_j)_+^r\,dxdt.
\end{align}
Similarly, we can derive
\begin{align}
\label{4.9}
\iint_{Q_j}(u-\widetilde{k}_j)_+^p\,dxdt 
& \leq \frac{\gamma 2^{(r-p)j}}{k^{r-p}} \iint_{Q_j}(u-k_j)_+^r\,dxdt,
\end{align}
and
\begin{align}
\label{4.10}
\iint_{Q_j}(u-\widetilde{k}_j)_+\,dxdt 
& \leq  \frac{\gamma 2^{(r-1)j}}{k^{r-1}} \iint_{Q_j}(u-k_j)_+^r\,dxdt.
\end{align}
Merging estimates \eqref{4.6}, \eqref{4.8}--\eqref{4.10} results to
\begin{align*}
& 2^{-j(1-q)_+} \operatorname*{ess \sup}_{-\widetilde{s}_j<t<0} \int_{\widetilde{K}_j \times\{t\}}(u-k_{j+1})_+^{q+1}\,dx+\iint_{\widetilde{Q}_j}|\nabla(u-\widetilde{k}_j)_+|^p\,dxdt \nonumber\\
\leq& \frac{\gamma 2^{j(N+p+r)}}{(1-\sigma)^{N+p} \sigma^p s} \frac{1}{k^{r-q-1}}\left(1+\frac{s}{\rho^p k^{q+1-p}}  +\frac{s}{\rho^p k^q}[\text{Tail}_{\infty}(u;0,\sigma\rho;-s,0)]^{p-1}\right)\\
&\times\iint_{Q_j}(u-k_j)_+^r\,dxdt,
\end{align*}
where we utilized the fact $\sigma\in(0,1)$.
Now, we choose $k$ to satisfy
\begin{align}
\label{4.11}
k \geq\left(\frac{s}{\rho^p}\right)^{\frac{1}{q+1-p}}+\left(\frac{s}{\rho^p}[\text{Tail}_\infty(u;0,\sigma\rho;-s,0)]^{p-1}\right)^{\frac{1}{q}},
\end{align}
so that
\begin{align}
\label{4.12}
& 2^{-j(1-q)_+} \operatorname*{ess \sup}_{-\widetilde{s}_j<t<0} \int_{\widetilde{K}_j \times\{t\}}(u-k_{j+1})_+^{q+1}\,dx+\iint_{\widetilde{Q}_j}|\nabla(u-\widetilde{k}_j)_+|^p\,dxdt \nonumber\\
\leq& \frac{\gamma 2^{j(N+p+r)}}{(1-\sigma)^{N+p}\sigma^p s} \frac{1}{k^{r-q-1}}\iint_{Q_j}(u-k_j)_+^r\,dxdt.
\end{align}
Since $m=p\frac{N+q+1}{N}$, we deduce from the Sobolev embedding Lemma \ref{lem-2-7} that
\begin{align*}
 \iint_{Q_{j+1}}&(u-k_{j+1})_+^m\,dxdt\leq \iint_{\widetilde{Q}_j}(u-k_{j+1})_+^m\,dxdt \nonumber\\
 \leq& \gamma \iint_{\widetilde{Q}_j}|\nabla(u-k_{j+1})_+|^p\,dxdt \left(\operatorname*{ess \sup}_{-\widetilde{s}_j<t<0} \int_{\widetilde{K}_j \times\{t\}}(u-k_{j+1})_+ ^{q+1}\,dx\right)^{\frac{p}{N}}\nonumber\\
\leq& \frac{ \gamma 2^{j\left[N+p+r+(1-q)_+ \right]\frac{N+p}{N}}}{(1-\sigma)^{\frac{(N+p)^2}{N}}\sigma^{\frac{p(N+p)}{N}} k^{(r-q-1) \frac{N+p}{N}}s^{\frac{N+p}{N}}}\left(\iint_{Q_j}(u-k_j)_+^r\,dxdt\right)^{\frac{N+p}{N}},
\end{align*}
where we used \eqref{4.12} to get the last line. The above estimate implies 
\begin{align}
\label{4.13}
& \mint\!\!\mint_{Q_{j+1}}\left(u-k_{j+1}\right)_+^m\,dxdt\nonumber\\
\leq&  \frac{ \gamma 2^{j\left[N+p+r+(1-q)_+ \right]\frac{N+p}{N}}}{(1-\sigma)^{\frac{(N+p)^2}{N}}\sigma^{\frac{p(N+p)}{N}} k^{(r-q-1) \frac{N+p}{N}}} \frac{\rho^p}{s}\left(\mint\!\!\mint_{Q_j}(u-k_j)_+^r\,dxdt\right)^{\frac{N+p}{N}},
\end{align}
where $\gamma$ depends only on $N,p,s,q,\Lambda$.

Note that inequality \eqref{4.13} holds for exponent $r\geq q+1$. Hereafter, we distinguish three cases: $q+1\leq r\leq m$; $r\leq m$ and $r<q+1$; $r>m$ to discuss the local boundedness of weak subsolutions to \eqref{1.1}. 
\subsubsection{The case $q+1\leq r\leq m$}
Denote
\begin{align*}
Y_j=\mint\!\!\mint_{Q_j}(u-k_j)_+^r\,dxdt.
\end{align*}
By H\"{o}lder's inequality and \eqref{4.7}, there exists a constant $\gamma$ depending only on $N,p,s,q,\Lambda$ such that 
\begin{align*}
Y_{j+1}&\leq \frac{1}{|Q_{j+1}|}\left(\iint_{Q_{j+1}}(u-k_{j+1})_+^m\,dxdt\right)^{\frac{r}{m}}|\widetilde{A}_j|^{1-\frac{r}{m}}\\
&\leq \gamma\left(\mint\!\!\mint_{Q_{j+1}}(u-k_{j+1})_+^m\,dxdt\right)^{\frac{r}{m}}\left(\frac{2^{(j+2)r}}{k^r} \mint\!\!\mint_{Q_j}(u-k_j)_+^r\,dxdt\right)^{1-\frac{r}{m}}.
\end{align*}
Based on the definition $\lambda_r=N(p-q-1)+rp$, it reads from \eqref{4.13} that
\begin{align*}
Y_{j+1}\leq \frac{\gamma b^j}{(1-\sigma)^{\frac{ r(N+p)^2}{N m}}\sigma^{\frac{rp(N+p)}{Nm}}}\left(\frac{\rho^p}{s}\right)^{\frac{r}{m}} k^{-\frac{r \lambda_r}{N m}} Y_j^{1+\frac{r p}{N m}},
\end{align*}
where $b:=2^{\frac{(N+p)r}{Nm}[N+p+r+(1-q)_{+}]+r}$ and $\gamma$ depends only on $N,p,s,q,\Lambda$. Moreover, Lemma \ref{lem-2-4} guarantees $Y_j\rightarrow 0$ as $j\rightarrow \infty$ if
\begin{align*}
Y_0=\mint\!\!\mint_{Q_0} u^r\,dxdt\leq \gamma^{-1}[\sigma(1-\sigma)]^{\frac{(N+p)^2}{p}}\left(\frac{s}{\rho^p}\right)^{\frac{N}{p}} k^{\frac{\lambda_r}{p}},
\end{align*}
which asks for enforcing
\begin{align}
\label{4.14}
k \geq \frac{\gamma}{[(1-\sigma)\sigma]^{\frac{(p+N)^2}{\lambda_r}}}\left(\frac{\rho^p}{s}\right)^{\frac{N}{\lambda_r}}\left(\mint\!\!\mint_{Q_0} u^r\,dxdt\right)^{\frac{p}{\lambda_r}} .
\end{align}

Combining the choices of $k$ in \eqref{4.11} and \eqref{4.14}, it follows form $Y_j\rightarrow 0$ that
\begin{align}
\label{4.15}
\operatorname*{ess \sup}_{Q_{\sigma \rho, \sigma s}}\ u \leq& \frac{\gamma}{[(1-\sigma)\sigma]^{\frac{(p+N)^2}{\lambda_r}}}\left(\frac{\rho^p}{s}\right)^{\frac{N}{\lambda_r}}\left(\mint\!\!\mint_{Q_{\rho,s}} u^r\,dxdt\right)^{\frac{p}{\lambda_r}}+\left(\frac{s}{\rho^p}\right)^{\frac{1}{q+1-p}}\nonumber\\
&+\left(\frac{s}{\rho^p}[\text{Tail}_\infty(u;0,\sigma\rho;-s,0)]^{p-1}\right)^{\frac{1}{q}},
\end{align}
where $\gamma$ depends only on $N,p,s,q,\Lambda$. With taking $\sigma=\frac{1}{2}$, we get the desired boundedness result for $q+1\leq r\leq m$.

\subsubsection{The case $r\leq m \text{ and } r<q+1$}

Since $r<q+1$, it is not hard to find $0<\lambda_r<\lambda_{q+1}$, from which we can deduce that $q+1<m$. Hence, we can exploit estimate \eqref{4.15} in the first case directly by replacing $q+1$ with $r$. For any $\sigma\in(0,1)$, we have
\begin{align}
\label{4.16}
\operatorname*{ess \sup}_{Q_{\sigma \rho, \sigma s}}\ u \leq& \frac{\gamma}{[(1-\sigma)\sigma]^{\frac{(p+N)^2}{\lambda_{q+1}}}}\left(\frac{\rho^p}{s}\right)^{\frac{N}{\lambda_{q+1}}}\left(\mint\!\!\mint_{Q_{\rho,s}} u^{q+1}\,dxdt\right)^{\frac{p}{\lambda_{q+1}}}+\left(\frac{s}{\rho^p}\right)^{\frac{1}{q+1-p}}\nonumber\\
&+\left(\frac{s}{\rho^p}[\text{Tail}_\infty(u;0,\sigma\rho;-s,0)]^{p-1}\right)^{\frac{1}{q}}.
\end{align}
Set
\begin{align*}
M_\sigma=\operatorname*{ess \sup}_{Q_{\sigma \rho, \sigma s}} \  u \quad \text { and } \quad M_1=\operatorname*{ess \sup}_{Q_{\rho,s}}\  u.
\end{align*}
Recalling the definition of $\lambda_{q+1}$, $\lambda_r$, we deduce from \eqref{4.16} that
\begin{align*}
M_\sigma \leq& M_1^{1-\frac{\lambda_r}{\lambda_{q+1}}} \frac{\gamma}{[(1-\sigma)\sigma]^{\frac{(p+N)^2}{\lambda_{q+1}}}}\left(\frac{\rho^p}{s}\right)^{\frac{N}{\lambda_{q+1}}}\left(\mint\!\!\mint_{Q_{\rho, s}} u^r\,dxdt\right)^{\frac{p}{\lambda_{q+1}}}+\left(\frac{s}{\rho^p}\right)^{\frac{1}{q+1-p}}\nonumber\\
&+\left(\frac{s}{\rho^p}[\text{Tail}_\infty(u;0,\sigma\rho;-s,0)]^{p-1}\right)^{\frac{1}{q}}.
\end{align*}
By Young's inequality, we have
\begin{align*}
M_\sigma \leq& \frac{1}{2} M_1+\frac{\gamma}{[(1-\sigma)\sigma]^{\frac{(p+N)^2}{\lambda_r}}}\left(\frac{\rho^p}{s}\right)^{\frac{N}{\lambda_r}}\left(\mint\!\!\mint_{Q_{\rho,s}} u^r\,dxdt\right)^{\frac{p}{\lambda_r}}+\left(\frac{s}{\rho^p}\right)^{\frac{1}{q+1-p}}\\
&+\left(\frac{s}{\rho^p}[\text{Tail}_\infty(u;0,\sigma\rho;-s,0)]^{p-1}\right)^{\frac{1}{q}}.
\end{align*}
Consider above estimate on the cylinders $Q_{\sigma_2 \rho, \sigma_2 s}$ and $Q_{\sigma_1 \rho, \sigma_1 s}$ with $\frac{1}{2}\leq \sigma_1\leq\sigma_2\leq 1$, it shows that
\begin{align*}
M_{\sigma_1} \leq& \frac{1}{2} M_{\sigma_2}+\frac{\gamma}{(\sigma_2-\sigma_1)^{\frac{(p+N)^2}{\lambda r}}}\left(\frac{\rho^p}{s}\right)^{\frac{N}{\lambda_r}}\left(\mint\!\!\mint_{Q_{\sigma_1 \rho, \sigma_1 s}} u^r\,dxdt \right)^{\frac{p}{\lambda_r}}+\left(\frac{s}{\rho^p}\right)^{\frac{1}{q+1-p}}\\
&+\left(\frac{s}{\rho^p}[\text{Tail}_\infty(u;0,\rho/2;-s,0)]^{p-1}\right)^{\frac{1}{q}}.
\end{align*}
This enables us to use iteration Lemma \ref{lem-2-5} to arrive at the claim in this case.	
\subsubsection{The case $r>m$}
According to the assumption $\lambda_r>0$, we can see $r>q+1$ from $r>m$. If not, there will hold $0<\lambda_r\leq\lambda_{q+1}=N(m-q-1)$, which implies that $q+1<m$, and this would yield a contradiction. In this case, we assume a prior that $u\in L_{\rm{loc}}^\infty(E_T)$ due to the embedding Lemma \ref{lem-2-7} generally not true. By iterative inequality \eqref{4.13} and the definition of $Y_j$, we have 
\begin{align*}
Y_{j+1}  =&\mint\!\!\mint_{Q_{j+1}}(u-k_{j+1})_+^r\,dxdt  \\
\leq&\|u\|_{\infty, Q_0}^{r-m} \mint\!\!\mint_{Q_{j+1}}(u-k_{j+1})_+^m\,dxdt \\
 \leq& \gamma\|u\|_{\infty, Q_0}^{r-m} \frac{\rho^p}{s} \frac{b^j}{[(1-\sigma)\sigma]^{\frac{(N+p)^2}{N}} k^{(r-q-1) \frac{N+p}{N}}}Y_j^{1+\frac{p}{N}},
\end{align*}
where $b=2^{\frac{(N+p)^2}{N}+r \frac{N+p}{N}+\frac{N+p}{N}(1-q)_{+}}$. Utilizing Lemma \ref{lem-2-4}, it holds that $Y_j\rightarrow 0$ as $j\rightarrow\infty$ if 
\begin{align*}
Y_0=\mint\!\!\mint_{Q_0} u^r\,dxdt \leq \gamma^{-1}\|u\|_{\infty, Q_0}^{-(r-m) \frac{N}{p}}[(1-\sigma)\sigma]^{\frac{(N+p)^2}{p}}\left(\frac{s}{\rho^p}\right)^{\frac{N}{p}} k^{\frac{(r-q-1)(N+p)}{p}},
\end{align*}
which requires us to choose $k$ fulling
\begin{align}
\label{4.17}
k \geq \frac{\gamma\|u\|_{\infty, Q_0}^{\frac{N(r-m)}{(N+p)(r-q-1)}}}{[(1-\sigma)\sigma]^{\frac{N+p}{r-q-1}}}\left(\frac{\rho^p}{s}\right)^{\frac{N}{(N+p)(r-q-1)}}\left(\mint\!\!\mint_{Q_0} u^r\,dxdt \right)^{\frac{p}{(N+p)(r-q-1)}}.
\end{align}

Taking the choices of $k$ in \eqref{4.11} and \eqref{4.17} into account, we deduce from $Y_j\rightarrow 0$ that
\begin{align*}
\operatorname*{ess \sup}_{Q_{\sigma \rho, \sigma s}}\ u \leq&  \frac{\gamma\|u\|_{\infty, Q_0}^{\frac{N(r-m)}{(N+p)(r-q-1)}}}{[(1-\sigma)\sigma]^{\frac{N+p}{r-q-1}}}\left(\frac{\rho^p}{s}\right)^{\frac{N}{(N+p)(r-q-1)}}\left(\mint\!\!\mint_{Q_{\rho,s}} u^r\,dxdt \right)^{\frac{p}{(N+p)(r-q-1)}} \\
&+\left(\frac{s}{\rho^p}\right)^{\frac{1}{q+1-p}}+\left(\frac{s}{\rho^p}[\text{Tail}_\infty(u;0,\sigma\rho;-s,0)]^{p-1}\right)^{\frac{1}{q}},
\end{align*}
where $\gamma$ depends only on $N,p,s,q,\Lambda$. 
We complete the proof by using the analogous argument as the second case $ r\leq m,r<q+1$.

\section{Integral-type Harnack inequality}
\label{sec5}
In this section, we will derive Theorem \ref{thm-1-5} which is a direct result from Theorem \ref{thm-1-2} and the following integral-type Harnack inequality

\begin{proposition}
\label{pro-5-1}
Let $0<p-1<q<p^2-1$, $\rho\in(0,1]$ and $K_{2\rho}(\bar{y})\times [s,\tau]\subset\subset E_T$. Assume that $u$ is a nonnegative weak solution to \eqref{1.1}. Then there exists a constant $\gamma>0$ depending only on $N,p,s,q,\Lambda$ such that
\begin{align*}
\sup _{t \in[s,\tau]} \int_{K_{\rho}(\bar{y})\times\{t\}} u^q\,dx \leq& \gamma \inf_{t \in[s,\tau]} \int_{K_{2\rho}(\bar{y})\times\{t\}} u^q\,dx +\gamma\left(\frac{\tau-s}{\rho^\lambda}\right)^{\frac{q}{q+1-p}}\\
&+\gamma\frac{\tau-s}{\rho^{p-N}}\left[\mathrm{Tail}_\infty(u;\bar{y},\rho;s,\tau)\right]^{p-1},
\end{align*}
where 
\begin{align*}
\lambda:=\frac{\lambda_q}{q} =\frac{N}{q}(p-q-1)+p.
\end{align*}
\end{proposition}
\begin{remark}
\label{rem-5-2}
If $u$ is only a nonnegative weak supersolution to \eqref{1.1}, there will hold
\begin{align*}
\sup _{t \in[s,\tau]} \int_{K_{\rho}(\bar{y})\times\{t\}} u^q\,dx \leq& \gamma\int_{K_{2\rho}(\bar{y})\times\{\tau\}} u^q\,dx +\gamma\left(\frac{\tau-s}{\rho^\lambda}\right)^{\frac{q}{q+1-p}}\\
&+\gamma\frac{\tau-s}{\rho^{p-N}}[\mathrm{Tail}_\infty(u;\bar{y},\rho;s,\tau)]^{p-1}.
\end{align*}
\end{remark}

Now, we are in a position to show a valuable estimate for proving Proposition \ref{pro-5-1}.

\begin{lemma}
\label{lem-5-3}
Let $0<p-1<q<p^2-1$, $\rho\in(0,1]$ and $K_\rho(\bar{y})\times[s,\tau]\subset\subset E_T$. Let $u$ be a nonnegative weak supersolution to \eqref{1.1}. For any $\sigma\in(0,1)$, there exists a constant $\gamma>0$ depending only on $N,p,s,q,\Lambda$ such that
\begin{align*}
& \int_{s}^{\tau}\int_{K_{\sigma \rho}(\bar{y}) }(t-s)^{\frac{1}{p}}(u+\kappa)^{-\frac{1+q}{p}}|\nabla u|^p\,dxdt\\
&+\!\int_{s}^{\tau}\!\int_{K_{\sigma \rho}(\bar{y})}\!\int_{K_{\sigma \rho}(\bar{y})}\!(t-s)^\frac{1}{p}(|u(x,t)+\kappa|+|u(y,t)+\kappa|)^{-\frac{1+q}{p}}\frac{|u(x,t)-u(y,t)|^p}{|x-y|^{N+sp}}\,dxdydt\\
\leq &\frac{\gamma \rho}{(1-\sigma)^{N+p}}\left(\frac{\tau-s}{\rho^\lambda}\right)^{\frac{1}{p}}\left(\sup _{t \in[s, \tau]} \int_{K_{\rho}(\bar{y}) \times\{t\}} u^q \,dx+\kappa^q \rho^N\right)^{\frac{(p-1)(q+1)}{p q}},
\end{align*}
where
\begin{align}
\label{5.1}
\lambda=\frac{N}{q}(p-q-1)+p \quad \textmd {and} \quad \kappa=\left(\frac{\tau-s}{\rho^p}\right)^{\frac{1}{q+1-p}}.
\end{align}
\end{lemma}
\begin{proof}
For simplicity, we may assume $(\bar{y},s)=(0,0)$. Choose
\begin{align*}
\varphi_h(x,t)=t^{\frac{1}{p}}\left([\![u]\!]_ {\bar{h}}+\kappa\right)^{-\frac{q+1-p}{p}} \zeta^p(x) \psi_{\varepsilon}(t) 
\end{align*}
as a test function in \eqref{1.7}, where $\zeta\in C_0^1\left(K_{\frac{\rho(1+\sigma)}{2}};[0,1]\right)$ satisfies $\zeta=1$ in $K_{\sigma\rho}$ and $|\nabla\zeta|\leq \frac{2}{(1-\sigma)\rho}$, for $\varepsilon>0$, $\psi_\varepsilon$ is a Lipschitz function such that $\psi_{\varepsilon}=1$ in $(\varepsilon, \tau-\varepsilon), \psi_{\varepsilon}=0$ outside $(0, \tau)$, and it is linearly interpolated otherwise. Then we yield that
\begin{align}
\label{5.2}
&\iint_{E_T}-u^q \partial_t \varphi_h\,dxdt+\iint_{E_T}|\nabla u|^{p-2}\nabla u\cdot\nabla \varphi_h\,dxdt\nonumber \\
&+\int_{0}^{T}\int_{\mathbb{R}^N}\int_{\mathbb{R}^N}|u(x,t)-u(y,t)|^{p-2}(u(x,t)-u(y,t))(\varphi_h(x,t)-\varphi_h(y,t))\,d\mu dt\geq 0.
\end{align}
The treatment of the first and second terms in \eqref{5.2} is similar to that of \cite[Lemma 7.3]{BDG23}, thus we have
\begin{align*}
-\iint_{E_T} u^q \partial_t \varphi_h\,dxdt\leq& q \tau^{\frac{1}{p}} \int_{K_{\rho} \times\{\tau\}} \zeta^p(x) \int_0^u s^{q-1}(s+\kappa)^{-\frac{q+1-p}{p}}\,dsdx
\end{align*}
and
\begin{align*}
\iint_{E_T}|\nabla u|^{p-2}\nabla u\cdot \nabla\varphi_h\,dxdt
\leq&-\frac{q+1-p}{2p} \int_{0}^{\tau}\int_{K_{\rho}}|\nabla u|^p (u+\kappa)^{-\frac{q+1}{p}} t^{\frac{1}{p}} \zeta^p(x)\,dxdt\\
&+\gamma \int_{0}^{\tau}\int_{K_{\rho}}(u+\kappa)^{\frac{p^2-1-q}{p}}|\nabla \zeta|^p t^{\frac{1}{p}}\,dxdt.
\end{align*} 
For the nonlocal term, we pass to the limit $h\rightarrow 0$ first and utilize Lemma \ref{lem-2-8} $(\rm{{\romannumeral 1}})$, and then send $\varepsilon\rightarrow 0$ leads to
\begin{align*}
&\int_{0}^{T}\int_{\mathbb{R}^N}\int_{\mathbb{R}^N}|u(x,t)-u(y,t)|^{p-2}(u(x,t)-u(y,t))(\varphi_h(x,t)-\varphi_h(y,t))\,d\mu dt\\
\rightarrow& \int_{0}^{\tau}\int_{\mathbb{R}^N}\int_{\mathbb{R}^N} U(x,y,t)t^{\frac{1}{p}}\left[(u(x,t)+\kappa)^{-\frac{q+1-p}{p}}\zeta^p(x)-(u(y,t)+\kappa)^{-\frac{q+1-p}{p}}\zeta^p(y)\right]\,d\mu dt\\
=&\int_{0}^{\tau}\int_{K_\rho}\int_{K_\rho}U(x,y,t)t^{\frac{1}{p}}\left[(u(x,t)+\kappa)^{-\frac{q+1-p}{p}}\zeta^p(x)-(u(y,t)+\kappa)^{-\frac{q+1-p}{p}}\zeta^p(y)\right]\,d\mu dt\\
&+2\int_{0}^{\tau}\!\!\int_{K_\rho}\!\int_{\mathbb{R}^N\backslash K_\rho}\!U(x,y,t)t^{\frac{1}{p}}\!\left[(u(x,t)+\kappa)^{-\frac{q+1-p}{p}}\zeta^p(x)\!-\!(u(y,t)+\kappa)^{-\frac{q+1-p}{p}}\zeta^p(y)\right]d\mu dt.
\end{align*}
Since $p-1<q<p^2-1$, it holds that $0<\frac{q+1-p}{p}<p-1$, thus we can employ Lemma \ref{lem-2-3} with $a=u(y,t)+\kappa$, $b=u(x,t)+\kappa$, $\tau_1=\zeta(y)$, $\tau_2=\zeta(x)$, and $\varepsilon=\frac{q+1-p}{p}$ to deduce
\begin{align*}
&U(x,y,t)\left[(u(x,t)+\kappa)^{-\frac{q+1-p}{p}}\zeta^p(x)-(u(y,t)+\kappa)^{-\frac{q+1-p}{p}}\zeta^p(y)\right]\\
\leq&- \gamma(p)\xi(\varepsilon)\left|\zeta(x)(u(x,t)+\kappa)^{\frac{p^2-q-1}{p^2}}-\zeta(y)(u(y,t)+\kappa)^{\frac{p^2-q-1}{p^2}}\right|^p\\
&+\left(\xi(\varepsilon)+1+\varepsilon^{-(p-1)}\right)|\zeta(x)-\zeta(y)|^p\left[(u(x,t)+\kappa)^{\frac{p^2-q-1}{p}}+(u(y,t)+\kappa)^{\frac{p^2-q-1}{p}}\right].
\end{align*}
Subsequently,
\begin{align*}
& \int_{0}^{\tau}\int_{K_\rho}\int_{K_\rho}U(x,y,t)t^{\frac{1}{p}}\left[(u(x,t)+\kappa)^{-\frac{q+1-p}{p}}\zeta^p(x)-(u(y,t)+\kappa)^{-\frac{q+1-p}{p}}\zeta^p(y)\right]\,d\mu dt\\
\leq&-\gamma\int_{0}^{\tau}\int_{K_\rho}\int_{K_\rho}t^{\frac{1}{p}}\frac{\left|\zeta(x)(u(x,t)+\kappa)^{\frac{p^2-q-1}{p^2}}-\zeta(y)(u(y,t)+\kappa)^{\frac{p^2-q-1}{p^2}}\right|^p}{|x-y|^{N+sp}}\,dxdydt\\
&+\gamma \int_{0}^{\tau}\int_{K_\rho}\int_{K_\rho}t^{\frac{1}{p}}\left[(u(x,t)+\kappa)^{\frac{p^2-q-1}{p}}+(u(y,t)+\kappa)^{\frac{p^2-q-1}{p}}\right]\frac{|\zeta(x)-\zeta(y)|^p}{|x-y|^{N+sp}}\,dxdydt,
\end{align*}
where $\gamma$ depends only on $N,p,s,q,\Lambda$. By the non-negativity of weak solutions, we estimate 
\begin{align*}
&\int_{0}^{\tau}\int_{K_\rho}\int_{\mathbb{R}^N\backslash K_\rho}U(x,y,t)t^{\frac{1}{p}}\left[(u(x,t)+\kappa)^{-\frac{q+1-p}{p}}\zeta^p(x)-(u(y,t)+\kappa)^{-\frac{q+1-p}{p}}\zeta^p(y)\right]\,d\mu dt\\
\leq&\gamma \int_{0}^{\tau}\int_{K_\rho}\int_{\mathbb{R}^N\backslash K_\rho\cap[u(x,t)\geq u(y,t)]}\frac{U(x,y,t)}{|x-y|^{N+sp}}t^{\frac{1}{p}}\zeta^p(x)(u(x,t)+\kappa)^{-\frac{q+1-p}{p}}\,dxdydt\\
\leq&\gamma\left(\sup_{x\in\rm{supp}\, \zeta(\cdot)}\int_{\mathbb{R}^N\backslash K_\rho}\frac{dy}{|x-y|^{N+sp}}\right)\int_{0}^{\tau}\int_{K_\rho}u(x,t)^{p-1}t^{\frac{1}{p}}\zeta^p(x)(u(x,t)+\kappa)^{-\frac{q+1-p}{p}}\,dxdt.
\end{align*}

Gathering the above estimates, we arrive at
\begin{align}
\label{5.3}
& \int_{0}^{\tau}\int_{K_{\rho}}|\nabla u|^p(u+\kappa)^{-\frac{q+1}{p}} t^{\frac{1}{p}} \zeta^p(x)\,dxdt\nonumber\\
&+\int_{0}^{\tau}\int_{K_\rho}\int_{K_\rho}t^{\frac{1}{p}}\frac{\left|\zeta(x)(u(x,t)+\kappa)^{\frac{p^2-q-1}{p^2}}-\zeta(y)(u(y,t)+\kappa)^{\frac{p^2-q-1}{p^2}}\right|^p}{|x-y|^{N+sp}}\,dxdydt\nonumber\\
 \leq& \gamma \tau^{\frac{1}{p}}\!\! \int_{K_{\rho} \times\{\tau\}} \zeta^p(x)\! \int_0^u s^{q-1}(s+\kappa)^{-\frac{q+1-p}{p}}\,dsdx\!+\!\frac{\gamma}{(1-\sigma)^p \rho^p} \int_{0}^{\tau}\!\!\int_{K_\rho} t^{\frac{1}{p}}(u+\kappa)^{\frac{p^2-1-q}{p}}\,dxdt\nonumber\\ 
&+\int_{0}^{\tau}\int_{K_\rho}\int_{K_\rho}t^{\frac{1}{p}}\left[(u(x,t)+\kappa)^{\frac{p^2-q-1}{p}}+(u(y,t)+\kappa)^{\frac{p^2-q-1}{p}}\right]\frac{|\zeta(x)-\zeta(y)|^p}{|x-y|^{N+sp}}\,dxdydt\nonumber\\
&+\gamma\left(\sup_{x\in\rm{supp}\, \zeta(\cdot)}\int_{\mathbb{R}^N\backslash K_\rho}\frac{dy}{|x-y|^{N+sp}}\right)\int_{0}^{\tau}\int_{K_\rho}u(x,t)^{p-1}t^{\frac{1}{p}}\zeta^p(x)(u(x,t)+\kappa)^{-\frac{q+1-p}{p}}\,dxdt\nonumber\\
=:&I_1+I_2+I_3+I_4.
\end{align}

\textbf{Estimate of $I_1$:} From the hypothesis $0<p-1<q$, one can easily get
\begin{align*}
\frac{(p-1)(q+1)}{pq} \in(0,1) \quad \text { and } \quad \int_0^u s^{q-1}(s+\kappa)^{-\frac{q+1-p}{p}}\,ds \leq \gamma u^{\frac{(p-1)(q+1)}{p}}.
\end{align*}	
Thus we apply H\"{o}lder's inequality to obtain
\begin{align}
\label{5.4}
I_1\leq\gamma \tau^{\frac{1}{p}} \int_{K_{\rho} \times\{\tau\}} u^{\frac{(p-1)(q+1)}{p}}\,dx \leq \gamma \rho\left(\frac{\tau}{\rho^\lambda}\right)^{\frac{1}{p}}\left(\sup_{t \in[0,\tau]} \int_{K_{\rho} \times\{t\}} u^q\,dx\right)^{\frac{(p-1)(q+1)}{p q}}.
\end{align}	

\textbf{Estimate of $I_2$:} Again using H\"{o}lder's inequality and the definition of $\kappa$, it follows that
\begin{align}
\label{5.5}
I_2
& \leq \frac{\gamma \tau^{\frac{1}{p}}}{(1-\sigma)^p} \frac{\tau}{\rho^p} \kappa^{-(q+1-p)} \sup_{t \in[0,\tau]} \int_{K_{\rho} \times\{t\}}(u+\kappa)^{\frac{(p-1)(q+1)}{p}}\,dx\nonumber \\
& \leq \frac{\gamma \rho}{(1-\sigma)^p}\left(\frac{\tau}{\rho^\lambda}\right)^{\frac{1}{p}}\left(\sup _{t \in[0,\tau]} \int_{K_{\rho} \times\{t\}} u^q\,dx+\kappa^q \rho^N\right)^{\frac{(p-1)(q+1)}{pq}}.
\end{align}

\textbf{Estimate of $I_3$:} By exchanging the role of $x$ and $y$ and by exploiting $\rho\in(0,1]$, we have
\begin{align}
\label{5.6}
I_3&\leq\frac{\gamma\rho^{p-sp}}{(1-\sigma)^p \rho^p} \int_{0}^{\tau}\int_{K_\rho} t^{\frac{1}{p}}(u+\kappa)^{\frac{p^2-1-q}{p}}\,dxdt\nonumber\\
&\leq \frac{\gamma \rho}{(1-\sigma)^p}\left(\frac{\tau}{\rho^\lambda}\right)^{\frac{1}{p}}\left(\sup _{t \in[0,\tau]} \int_{K_{\rho} \times\{t\}} u^q\,dx+\kappa^q \rho^N\right)^{\frac{(p-1)(q+1)}{pq}}.
\end{align}

\textbf{Estimate of $I_4$:} Notice that 
\begin{align*}\frac{|y|}{|y-x|}\leq 1+\frac{|x|}{|y-x|}\leq 1+\frac{1+\sigma}{1-\sigma}\leq\frac{2}{1-\sigma}
\end{align*}
for any  $|x|\leq \frac{(1+\sigma)\rho}{2}$ and $|y|\geq\rho$, thus there holds
\begin{align}
\label{5.7}
I_4&\leq \frac{\gamma 2^{N+sp}}{(1-\sigma)^{N+sp}\rho^{sp}}\int_{0}^{\tau}\int_{K_{\rho}} t^\frac{1}{p} (u(x,t)+\kappa)^{\frac{p^2-1-q}{p}}\,dxdt\nonumber\\
&\leq \frac{\gamma 2^{N+sp}\rho}{(1-\sigma)^{N+sp}}\left(\frac{\tau}{\rho^\lambda}\right)^{\frac{1}{p}}\left(\sup _{t \in[0,\tau]} \int_{K_{\rho} \times\{t\}} u^q\,dx+\kappa^q \rho^N\right)^{\frac{(p-1)(q+1)}{pq}}.
\end{align}

On the other hand, the left-hand side of \eqref{5.3} can be estimated by utilizing Lemma \ref{lem-2-1} with $\alpha=\frac{p^2-q-1}{p^2}>0$ and the property $\zeta=1$ in $K_{\sigma\rho}$, it gives 
\begin{align}
\label{5.8}
&\int_{0}^{\tau}\int_{K_\rho}\int_{K_\rho}t^{\frac{1}{p}}\frac{\left|\zeta(x)(u(x,t)+\kappa)^{\frac{p^2-q-1}{p^2}}-\zeta(y)(u(y,t)+\kappa)^{\frac{p^2-q-1}{p^2}}\right|^p}{|x-y|^{N+sp}}\,dxdydt\nonumber\\
\geq&\int_{0}^{\tau}\int_{K_{\sigma\rho}}\int_{K_{\sigma\rho}}t^{\frac{1}{p}}\frac{\left|(u(x,t)+\kappa)^{\frac{p^2-q-1}{p^2}}-(u(y,t)+\kappa)^{\frac{p^2-q-1}{p^2}}\right|^p}{|x-y|^{N+sp}}\,dxdydt\nonumber\\
\geq&\int_{0}^{\tau}\int_{K_{\sigma\rho}}\int_{K_{\sigma\rho}}t^{\frac{1}{p}}\left(|u(x,t)+\kappa|+|u(y,t)+\kappa|\right)^{-\frac{q+1}{p}}\frac{|u(x,t)-u(y,t)|^p}{|x-y|^{N+sp}}\,dxdydt.
\end{align}
Combining estimates \eqref{5.3}--\eqref{5.8}, we get the conclusion.
\end{proof}
Given Lemma \ref{lem-5-3}, we can explore the following result.
\begin{lemma}
\label{lem-5-4}
Let $0<p-1<q<p^2-1$, $\rho\in(0,1]$ and $K_\rho(\bar{y})\times[s,\tau]\subset\subset E_T$. Assume that $u$ is a nonnegative weak supersolution to \eqref{1.1}. For all $\delta, \sigma\in(0,1)$, there exists a constant $\gamma>0$ depending only on $N,p,s,q,\Lambda$ such that
\begin{align*}
& \frac{1}{\rho} \int_{s}^{\tau}\int_{K_{\sigma \rho}(\bar{y})}|\nabla u|^{p-1}\,dxdt+\frac{1}{\rho}\int_{s}^{\tau}\int_{K_{\sigma \rho}(\bar{y})}\int_{K_{\sigma \rho}(\bar{y})}\frac{|u(x,t)-u(y,t)|^{p-1}}{|x-y|^{N+sp-1}}\,dxdydt\\
& \quad \leq \delta \sup _{t \in[s, \tau]} \int_{K_{\rho}(\bar{y}) \times\{t\}} u^q \,dx+\frac{\gamma}{\left[\delta^{q+1}(1-\sigma)^{p q}\right]^{\frac{N+p}{q+1-p}}}\left(\frac{\tau-s}{\rho^\lambda}\right)^{\frac{q}{q+1-p}},
\end{align*}
where $\lambda$ and $\kappa$ defined in \eqref{5.1}.
\end{lemma}
\begin{proof}
Let $(\bar{y},s)=(0,0)$. We derive from H\"{o}lder's inequality that
\begin{align}
\label{5.9}
\int_{0}^{\tau}\int_{K_{\sigma \rho}}|\nabla u|^{p-1}\,dxdt \leq& \left(\int_{0}^{\tau}\int_{K_{\sigma \rho}}|\nabla u|^p(u+\kappa)^{-\frac{q+1}{p}} t^{\frac{1}{p}}\,dxdt\right)^{\frac{p-1}{p}} \nonumber\\
& \times\left(\int_{0}^{\tau}\int_{K_{\sigma \rho}}(u+\kappa)^{\frac{(p-1)(q+1)}{p}} t^{\frac{1-p}{p}}\,dxdt\right)^{\frac{1}{p}}. 
\end{align}
The estimate of the first integral on the right-hand side of \eqref{5.9} yields from Lemma \ref{lem-5-3}. Coming to estimate the second integral, there holds by using H\"{o}lder's inequality that
\begin{align}
\label{5.10}
& \int_{0}^{\tau}\int_{K_{\sigma\rho}}(u+\kappa)^{\frac{(p-1)(q+1)}{p}} t^{\frac{1-p}{p}}\,dxdt\nonumber\\
& \leq \int_0^\tau t^{\frac{1-p}{p}}\,dt \times \sup _{t \in[0, \tau]} \int_{K_{\sigma \rho} \times\{t\}}(u+\kappa)^{\frac{(p-1)(q+1)}{p}}\,dx\nonumber\\
& \leq \gamma \rho\left(\frac{\tau}{\rho^\lambda}\right)^{\frac{1}{p}}\left(\sup_{t \in[0,\tau]} \int_{K_{\sigma \rho} \times\{t\}} u^q\,dx+\kappa^q \rho^N\right)^{\frac{(p-1)(q+1)}{p q}}.
\end{align}
Still by H\"{o}lder's inequality, we compute
\begin{align}
\label{5.11}
&\int_{0}^{\tau}\int_{K_{\sigma \rho}}\int_{K_{\sigma \rho}}\frac{|u(x,t)-u(y,t)|^{p-1}}{|x-y|^{N+sp-1}}\,dxdydt\nonumber\\
\leq& \left(\int_{0}^{\tau}\int_{K_{\sigma \rho}}\int_{K_{\sigma \rho}}\frac{|u(x,t)-u(y,t)|^p}{|x-y|^{N+sp}}(u(x,t)+\kappa)^{-\frac{q+1}{p}} t^{\frac{1}{p}}\,dxdydt\right)^{\frac{p-1}{p}}\nonumber\\
&\times\left(\int_{0}^{\tau}\int_{K_{\sigma \rho}}\int_{K_{\sigma \rho}}\frac{(u(x,t)+\kappa)^{\frac{(p-1)(q+1)}{p}}}{|x-y|^{N+sp-p}}t^{\frac{1-p}{p}}\,dxdydt\right)^{\frac{1}{p}}.
\end{align}
The first integral on the right-hand side of \eqref{5.11} has been provided in Lemma \ref{lem-5-3}. We deal with the second integral as in \eqref{5.10} and obtain
\begin{align*}
&\int_{0}^{\tau}\int_{K_{\sigma \rho}}\int_{K_{\sigma \rho}}\frac{(u(x,t)+\kappa)^{\frac{(p-1)(q+1)}{p}}}{|x-y|^{N+sp-p}}t^{\frac{1-p}{p}}\,dxdydt\\
\leq &\gamma \rho\left(\frac{\tau}{\rho^\lambda}\right)^{\frac{1}{p}}\left(\sup_{t \in[0,\tau]} \int_{K_{\sigma \rho} \times\{t\}} u^q\,dx+\kappa^q \rho^N\right)^{\frac{(p-1)(q+1)}{p q}},
\end{align*}
where we used the assumption $\sigma\in(0,1)$, and the radius $\rho\in(0,1]$. Putting the above estimates together, employing Young's inequality as well as the definition of $\kappa$ in \eqref{5.1}, we infer that
\begin{align*}
& \int_{0}^{\tau}\int_{K_{\sigma \rho}}|\nabla u|^{p-1}\,dxdt+\int_{0}^{\tau}\int_{K_{\sigma \rho}}\int_{K_{\sigma \rho}}\frac{|u(x,t)-u(y,t)|^{p-1}}{|x-y|^{N+sp-1}}\,dxdydt \\
& \leq \frac{\gamma \rho}{(1-\sigma)^{N+p}}\left(\frac{\tau}{\rho^\lambda}\right)^{\frac{1}{p}}\left(\sup _{t \in[0, \tau]} \int_{K_{\rho} \times\{t\}} u^q \,dx+\kappa^q \rho^N\right)^{\frac{(p-1)(q+1)}{p q}} \\
& \leq \delta \rho \sup _{t \in[0, \tau]} \int_{K_{\rho} \times\{t\}} u^q\,dx +\frac{\gamma \rho}{\left[\delta^{q+1}(1-\sigma)^{p q}\right]^{\frac{N+p}{q+1-p}}}\left(\frac{\tau}{\rho^\lambda}\right)^{\frac{q}{q+1-p}},
\end{align*}
where the constant $\delta\in(0,1)$ depends on $p,q$. Diving both sides by $\rho$, we arrive at the claim.
\end{proof}

Hereafter, we give the proof of Proposition \ref{pro-5-1}.

\begin{proof}[\textbf{Proof of Proposition \ref{pro-5-1}}]
Assume that $(\bar{y},s)=(0,0)$. Denote
\begin{align*}
\rho_j=\sum_{n=0}^j \frac{\rho}{2^n}, \quad \widetilde{\rho}_j=\frac{\rho_j+\rho_{j+1}}{2}, \quad \widehat{\rho}_j=\frac{3\rho_j+\rho_{j+1}}{4}, \ \ j=0,1,2\ldots
\end{align*}
and
\begin{align*}
\quad K_j=K_{\rho_j}, \quad \widetilde{K}_j=K_{\widetilde{\rho}_j}, \quad \widehat{K}_j=K_{\widehat{\rho}_j}, \ \ j=0,1,2\ldots.
\end{align*}
Consider the function $\zeta\in C_0^1(\widetilde{K}_j;[0,1])$ that vanishes outside $\widehat{K}_j$, equals to $1$ in $K_j$ such that $|\nabla\zeta|\leq 2^{j+3}/\rho$. Testing \eqref{1.7} with $\zeta$, there holds for any $t_1, t_2 \in[0, \tau]$ that
\begin{align}
\label{5.12}
\int_{\widetilde{K}_j \times\{t_1\}} u^q \zeta\,dx \leq& \int_{\widetilde{K}_j \times\{t_2\}} u^q \zeta\,dx+\frac{2^{j+3}}{\rho} \int_{t_1}^{t_2}\int_{\widetilde{K}_j}|\nabla u|^{p-1}\,dxdt\nonumber\\
&+\frac{2^{j+3}}{\rho}\int_{t_1}^{t_2}\int_{\widetilde{K}_j}\int_{\widetilde{K}_j}\frac{|u(x,t)-u(y,t)|^{p-1}}{|x-y|^{N+sp-1}}\,dxdydt\nonumber\\
&+2\int_{t_1}^{t_2}\int_{\widetilde{K}_j}\int_{\mathbb{R}^N\backslash\widetilde{K}_j}\frac{|u(x,t)-u(y,t)|^{p-1}}{|x-y|^{N+sp}}\zeta(x)\,dxdydt.
\end{align}
Through direct computation,
\begin{align*}
\frac{|y|}{|y-x|}\leq 1+\frac{|x|}{|y-x|}\leq 1+\frac{\widehat{\rho}_j}{\widetilde{\rho}_j-\widehat{\rho}_j}\leq 2^{j+3}
\end{align*}
for any  $|x|\leq\widehat{\rho}_j$ and $|y|\geq\widetilde{\rho}_j$, thus the last term in \eqref{5.12} can be estimated as
\begin{align}
\label{5.13}
&\int_{t_1}^{t_2}\int_{\widetilde{K}_j}\int_{\mathbb{R}^N\backslash\widetilde{K}_j}\frac{|u(x,t)-u(y,t)|^{p-1}}{|x-y|^{N+sp}}\zeta(x)\,dxdydt\nonumber\\
\leq& \frac{\tau 2^{(j+3)(N+sp)}}{\rho^{p-N}} [\text{Tail}_\infty(u;0,\rho;0,\tau)]^{p-1}.
\end{align}
For $t_1>t_2$, we may test \eqref{1.7} by $-\zeta$ instead of $\zeta$. To proceed, we set
\begin{align*}
\int_{K_{2 \rho} \times\{t_2\}} u^q\,dx=\inf_{t\in[0,\tau]} \int_{K_{2\rho} \times\{t\}} u^q\,dx=: A .
\end{align*}
Define 
\begin{align*}
S_j:=\sup_{t\in[0,\tau]} \int_{K_j \times\{t\}} u^q \,dx.
\end{align*}
Since $t_1\in[0,\tau]$ is arbitrary, it tells from \eqref{5.12} and \eqref{5.13} that
\begin{align}
\label{5.14}
S_j \leq& A+\frac{2^{j+3}}{\rho}\int_{0}^{\tau} \int_{\widetilde{K}_j} |\nabla u|^{p-1}\,dxdt+\frac{2^{j+3}}{\rho}\int_{0}^{\tau}\int_{\widetilde{K}_j}\int_{\widetilde{K}_j}\frac{|u(x,t)-u(y,t)|^{p-1}}{|x-y|^{N+sp-1}}\,dxdydt\nonumber\\
&+\frac{\tau 2^{(j+4)(N+sp)}}{\rho^{p-N}} [\text{Tail}_\infty(u;0,\rho;0,\tau)]^{p-1}.
\end{align} 
From the definition of $\rho_j$, we can get $\rho\leq\rho_j<2\rho$. Selecting $\sigma$ small enough such that $1-\sigma\geq 2^{-(j+4)}$, and we deduce from Lemma \ref{lem-5-4} that
\begin{align*}
& \frac{1}{2 \rho} \int_{0}^{\tau} \int_{\widetilde{K}_j}|\nabla u|^{p-1}\,dxdt+\frac{1}{2 \rho}
\int_{0}^{\tau}\int_{\widetilde{K}_j}\int_{\widetilde{K}_j}\frac{|u(x,t)-u(y,t)|^{p-1}}{|x-y|^{N+sp-1}}\,dxdydt\\
\leq& \frac{1}{\rho_{j+1}}\int_{0}^{\tau} \int_{\widetilde{K}_j} |\nabla u|^{p-1}\,dxdt+\frac{1}{\rho_{j+1}}\int_{0}^{\tau}\int_{\widetilde{K}_j}\int_{\widetilde{K}_j}\frac{|u(x,t)-u(y,t)|^{p-1}}{|x-y|^{N+sp-1}}\,dxdydt\\
\leq& \delta \sup_{t \in[0,\tau]} \int_{K_{j+1} \times\{t\}} u^q\,dx +\frac{\gamma}{\left[\delta^{q+1}(1-\sigma)^{p q}\right]^{\frac{N+p}{q+1-p}}}\left(\frac{\tau}{\rho_{j+1}^\lambda}\right)^{\frac{q}{q+1-p}} \\
\leq& \delta S_{j+1}+\frac{\gamma 2^{j \frac{p q(N+p)}{q+1-p}}}{\delta^{\frac{(q+1)(N+p)}{q+1-p}}}\left(\frac{\tau}{\rho^\lambda}\right)^{\frac{q}{q+1-p}}.
\end{align*}
Multiplying both sides of the above inequality by $2^{j+4}$ leads to 
\begin{align*}
& \frac{2^{j+3}}{\rho} \int_{0}^{\tau} \int_{\widetilde{K}_j}|\nabla u|^{p-1}\,dxdt+\frac{2^{j+3}}{\rho}
\int_{0}^{\tau}\int_{\widetilde{K}_j}\int_{\widetilde{K}_j}\frac{|u(x,t)-u(y,t)|^{p-1}}{|x-y|^{N+sp-1}}\,dxdydt\\
\leq& \delta 2^{j+4} S_{j+1}+\frac{\gamma 2^{j\left[1+\frac{p q(N+p)}{q+1-p}\right] }}{\delta^{\frac{(q+1)(N+p)}{q+1-p}}}\left(\frac{\tau}{\rho^\lambda}\right)^{\frac{q}{q+1-p}}.
\end{align*}
For some $\varepsilon\in(0,1)$, letting $\delta=\varepsilon/2^{j+4}$ to get
\begin{align}
\label{5.15}
&\frac{ 2^{j+3}}{\rho} \int_{0}^{\tau}\int_{\widetilde{K}_j}|\nabla u|^{p-1}\,dxdt+\frac{2^{j+3}}{\rho}
\int_{0}^{\tau}\int_{\widetilde{K}_j}\int_{\widetilde{K}_j}\frac{|u(x,t)-u(y,t)|^{p-1}}{|x-y|^{N+sp-1}}\,dxdydt\nonumber\\
\leq& \varepsilon S_{j+1}+\gamma(\varepsilon) b^j\left(\frac{\tau}{\rho^\lambda}\right)^{\frac{q}{q+1-p}}
\end{align}
with $b=b(p,q,N)>1$. By virtue of \eqref{5.14} and \eqref{5.15}, we have for any $ j \in \mathbb{N}\cup \{0\}$ that
\begin{align}
\label{5.16}
S_j \leq \varepsilon S_{j+1}+\gamma(\varepsilon) b^j\left(A+\left(\frac{\tau}{\rho^\lambda}\right)^{\frac{q}{q+1-p}}+\tau \rho^{N-p}[\text{Tail}_\infty(u;0,\rho;0,\tau)]^{p-1}\right).
\end{align}
Iterating inequality \eqref{5.16} gives that
\begin{align*}
S_0 \leq \varepsilon^j S_j+\gamma(\varepsilon)\left(A+\left(\frac{\tau}{\rho^\lambda}\right)^{\frac{q}{q+1-p}}+\tau \rho^{N-p}[\text{Tail}_\infty(u;0,\rho;0,\tau)]^{p-1}\right) \sum_{i=0}^{j-1}(\varepsilon b)^i .
\end{align*}
Then, by letting $\varepsilon=1/2b$, we get $\sum_{i=0}^{j-1}(\varepsilon b)^i\leq 2$, and meanwhile $\varepsilon^j S_j \rightarrow 0$ as $j \rightarrow \infty$. The proof is completed by recalling the definitions of $S_j$ and $A$.
\end{proof}

Finally, Theorem \ref{thm-1-5} follows from a combination of Theorem \ref{thm-1-2} and Proposition \ref{pro-5-1}.

\section{Expansion of positivity}
\label{sec6}

This section aims to obtain the following result regarding the expansion of positivity, which derives the pointwise estimate for weak supersolutions from the measure theoretical condition.

\begin{proposition}
\label{pro-6-1}
Let $0<p-1\leq q$. Assume that $u\in L^\infty(\mathbb {R}^N\times(0,T))$ is a nonnegative, weak supersolution to \eqref{1.1}. If for some constants $M>0$, $\alpha \in(0,1)$ we have
\begin{align}
\label{6.1}
|[u(\cdot, t_0) \geq M] \cap K_{\rho}(x_0)| \geq \alpha|K_{\rho}|,
\end{align}
then there exist parameters $\delta, \eta \in(0,1)$ depending only on $N,p,s,q,\Lambda, \alpha$ such that
\begin{align*}
u \geq \eta M \quad \text { a.e. in } K_{2 \rho}(x_0) \times\left(t_0+\tfrac{1}{2} \delta M^{q+1-p} \rho^p, t_0+\delta M^{q+1-p} \rho^p\right],
\end{align*}
provided
\begin{align*}
K_{8 \rho}(x_0) \times(t_0, t_0+\delta M^{q+1-p} \rho^p] \subset E_T.
\end{align*}

\end{proposition}

Before proving Proposition \ref{pro-6-1}, we give several preparatory estimates. The first one is a De Giorgi-type lemma.

\begin{lemma}
\label{lem-6-2}
Let $q>0$, $p>1$, and let $u\in L^\infty(\mathbb{R}^N\times(0,T))$ be a nonnegative, weak supersolution to \eqref{1.1}. For some constants $M>0$ and $\delta\in(0,1)$, assume that $\rho\in(0,1]$ and $(x_0,t_0)+Q_\rho(\theta)\subset E_T$, where $\theta=\delta M^{q+1-p}$. If there is a constant $\nu\in(0,1)$ only depending on $N, p, s, q, \Lambda$ and $\delta$ such that
\begin{align*}
\left|[u \leq M] \cap(x_0, t_0)+Q_{\rho}(\theta)\right| \leq \nu|Q_{\rho}(\theta)|,
\end{align*}
then we  have
\begin{align*}
u \geq \frac{1}{2} M  \quad \text {a.e. in }(x_0, t_0)+Q_{\frac{1}{2} \rho}(\theta).
\end{align*}
\end{lemma}
\begin{proof}
Let $(x_0,t_0)=(0,0)$. Take decreasing sequences
\begin{align*}
k_j=\frac{M}{2}+\frac{M}{2^{j+1}}, \ \ j=0,1,2\ldots.
\end{align*}
Denote
\begin{align*}
\rho_j=\frac{\rho}{2}+\frac{\rho}{2^{j+1}}, \quad \widetilde{\rho}_j=\frac{\rho_j+\rho_{j+1}}{2}, \quad  \widehat{\rho}_j=\frac{3\rho_j+\rho_{j+1}}{4}, \ \ j=0,1,2\ldots.
\end{align*}
Set the domains
\begin{align*}
K_j=K_{\rho_j}, \quad \widetilde{K}_j=K_{\widetilde{\rho}_j}, \quad \widehat{K}_j=K_{\widehat{\rho}_j}, \ \ j=0,1,2\ldots,
\end{align*}
and
\begin{align*}
Q_j=K_j \times\left(-\theta \rho_j^p, 0\right], \quad \widetilde{Q}_j=\widetilde{K}_j \times\left(-\theta \widetilde{\rho}_j^p, 0\right],\quad
\widehat{Q}_j=\widehat{K}_j \times\left(-\theta \widehat{\rho}_j^p, 0\right], \ \ j=0,1,2\ldots.
\end{align*}
Consider the cutoff function $0\leq\zeta\leq 1$ in $Q_j$ vanishing outside $\widehat{Q}_j$, and equals to $1$ in $\widetilde{Q}_j$, such that
\begin{align*}
|\nabla\zeta|\leq\frac{2^{j+4}}{\rho} \quad \text { and } \quad |\partial_t \zeta| \leq \frac{2^{p(j+4)}}{\theta \rho^p} .
\end{align*}
An application of the energy estimate \eqref{4.1} and Lemma \ref{lem-2-2} in this setting leads to
\begin{align}
\label{6.2}
& \operatorname*{ess\sup}_{-\theta \rho_j^p<t<0} \int_{K_j} \zeta^p(u+k_j)^{q-1}(u-k_j)_{-}^2\,dx+\iint_{Q_j} \zeta^p|\nabla(u-k_j)_-|^p\,dxdt \nonumber\\
\leq&\gamma \iint_{Q_j}(u+k_j)^{q-1}(u-k_j)_-^2|\partial_t \zeta|\,dxdt+\gamma \iint_{Q_j}(u-k_j)_-^p|\nabla \zeta|^p\,dxdt\nonumber\\
&+\gamma\int_{-\theta\rho_j^p}^{0} \int_{K_j} \int_{K_j} \max \{(u-k_j)_-(x, t),(u-k_j)_-(y,t)\}^p|\zeta(x,t)-\zeta(y,t)|^p\,d\mu dt \nonumber\\
&+\gamma \mathop{\mathrm{ess}\,\sup}_{\stackrel{-\theta\rho_j^p<t<0}{x\in \mathrm{supp}\,\zeta(\cdot,t)}}\int_{\mathbb{R}^N \backslash K_j} \frac{(u-k_j)_-^{p-1}(y,t)}{|x-y|^{N+sp}}\,dy\iint_{Q_j}(u-k_j)_-\zeta^p(x,t)\,dxdt\nonumber\\
 =:& G_1+G_2+G_3+G_4.
\end{align}

For the first term $G_1$, observe that when $u<k_j$, there holds
\begin{align}
\label{6.3}
\frac{1}{2}M\leq k_j\leq u+k_j\leq 2M.
\end{align}
Thus we have
\begin{align*}
G_1 &\leq \gamma \frac{2^{pj}}{\theta \rho^p} \iint_{Q_j}(u+k_j)^{q-1}(u-k_j)_{-}^2\,dxdt\\
&\leq \gamma \frac{2^{pj}}{\theta \rho^p} M^{q+1}|A_j|,
\end{align*}
where $A_j=[u<k_j]\cap Q_j$. With the properties of function $\zeta$, we get the estimates of $G_2$ and $G_3$ as below
\begin{align*}
G_2,G_3\leq\gamma\frac{2^{pj}}{\rho^p} M^p|A_j|.
\end{align*}
Besides, we apparently see
\begin{align*}
\frac{|y|}{|y-x|}\leq 1+\frac{|x|}{|y-x|}\leq 1+\frac{\widehat{\rho}_j}{\rho_j-\widehat{\rho}_j}\leq 2^{j+4}
\end{align*}
for any $|x|\leq\widehat{\rho}_j$ and  $|y|\geq\rho_j$. Hence, it holds that
\begin{align*}
G_4&\leq\gamma 2^{(j+4)(N+sp)} M^p|A_j|\left(\frac{\gamma 2^{spj}}{\rho^{sp}}+\operatorname*{ess \sup}_{-\theta\rho_j^{sp}<t<0}\int_{\mathbb{R}^N\backslash K_\rho}\frac{1}{|y|^{N+sp}}\,dy\right)\\ 
&\leq \gamma 2^{(j+4)(N+sp)}\frac{M^p}{\rho^p}|A_j|.
\end{align*}

Now, let us turn to estimate the left-hand side of \eqref{6.2}. By using \eqref{6.3} one can get
\begin{align*}
 \operatorname*{ess\sup}_{-\theta \rho_j^p<t<0} \int_{K_j} \zeta^p(u+k_j)^{q-1}(u-k_j)_{-}^2\,dx\geq\frac{M^{q-1}}{2^{|q-1|}}\operatorname*{ess\sup}_{-\theta \rho_j^p<t<0} \int_{K_j}\zeta^p(u-k_j)_-^2\,dx.
\end{align*}
We conclude from the above estimates that
\begin{align}
\label{6.4}
&\frac{M^{q-1}}{2^{|q-1|}}\operatorname*{ess\sup}_{-\theta \widetilde{\rho}_j^p<t<0} \int_{\widetilde{K}_j}(u-k_j)_-^2\,dx+\iint_{\widetilde{Q}_j}|\nabla(u-k_j)_-|^{p}\,dxdt\nonumber\\
\leq& \gamma \frac{2^{(N+p)j}}{\rho^p} M^p\left(1+\frac{M^{q+1-p}}{\theta}\right)|A_j|. 
\end{align}
H\"{o}lder's inequality in conjunction with Sobolev embedding Lemma \ref{lem-2-7} indicates that
\begin{align}
\label{6.5}
\frac{M}{2^{j+4}}& |A_{j+1}| \leq \iint_{\widetilde{Q}_j}(u-k_j)_{-} \,dxdt\nonumber\\
\leq&\left(\iint_{\widetilde{Q}_j}(u-k_j)_{-}^{p \frac{N+2}{N}}\,dxdt\right)^{\frac{N}{p(N+2)}}|A_j|^{1-\frac{N}{p(N+2)}}\nonumber \\
\leq& \gamma\left(\iint_{\widetilde{Q}_j}\left|\nabla(u-k_j)_{-} \right|^p\,dxdt\right)^{\frac{N}{p(N+2)}}\left(\operatorname*{ess \sup}_{-\theta \widetilde{\rho}_j^p<t<0} \int_{\widetilde{K}_j}(u-k_j)_{-}^2\,dx\right)^{\frac{1}{N+2}}|A_j|^{1-\frac{N}{p(N+2)}}\nonumber \\
\leq &\gamma\left[\frac{2^{(N+p) j}}{\rho^p} M^p\left(1+\frac{M^{q+1-p}}{\theta}\right)\right]^{\frac{N+p}{p(N+2)}}\left(\frac{2^{|q-1|}}{M^{q-1}}\right)^{\frac{1}{N+2}}|A_j|^{1+\frac{1}{N+2}},
\end{align}
where we used \eqref{6.4} in the last step. Denote $Y_j=|A_j|/|Q_j|$, we infer from \eqref{6.5} that
\begin{align*}
Y_{j+1} \leq \gamma b^j\left(1+\frac{M^{q+1-p}}{\theta}\right)^{\frac{N+p}{p(N+2)}}\left(\frac{\theta}{M^{q+1-p}}\right)^{\frac{1}{N+2}} Y_j^{1+\frac{1}{N+2}},
\end{align*}
where $b=b(N,s,p)>1$, and $\gamma> 0$ depends only on $N,p,s,q,\Lambda$. With the help of Lemma \ref{lem-2-4}, we get $Y_j\rightarrow 0$ as $j\rightarrow\infty$ if $Y_0\leq \nu$. Recalling the definition of $\theta$, we find $\nu$ only depends on $N,p,s,q,\Lambda$ and $\delta$.
\end{proof}

The following result is about extending the measure information of positivity forward in time direction.

\begin{lemma}
	\label{lem-6-3}
Let $q>0,p>1$ and $\rho\in(0,1]$. Suppose that $u$ is a nonnegative weak solution to \eqref{1.1}. If for some $M>0$ and $\alpha\in(0,1)$ there holds
\begin{align*}
|[u(\cdot, t_0)\geq M] \cap K_{\rho}(x_0)| \geq \alpha|K_{\rho}|,
\end{align*}
then there exist constants $\delta$ and $\varepsilon$ in $(0,1)$ depending on $N,p,s,q,\Lambda$ and $\alpha$ such that
\begin{align*}
|[u(\cdot, t)] \geq \varepsilon M] \cap K_{\rho}(x_0)| \geq \frac{1}{2} \alpha|K_{\rho}| \quad \text { for all } t \in\left(t_0, t_0+\delta M^{q+1-p} \rho^p\right],
\end{align*}
provided 
\begin{align*}
Q:=K_{\rho} \times(t_0,t_0+\delta M^{q+1-p} \rho^p]\subset E_T.
\end{align*}
\end{lemma}
\begin{proof}
Let $(x_0,t_0)=(0,0)$. Employing the energy estimate \eqref{4.1} in $Q$ with $k=M$. Choose the test function $\zeta(x,t)=\zeta(x)$ equals to $1$ in $K_{(1-\sigma)\rho}$, vanishing outside 
$K_{\frac{2-\sigma}{2}\rho}$ such that $|\nabla\zeta|\leq (\sigma\rho)^{-1}$,
where $\sigma\in(0,1)$ will be chosen later. Then we have for any $0<t\leq \delta M^{q+1-p}\rho^p$ that
\begin{align*}
&\int_{K_{\rho}\times\{t\}}\int_u^M s^{q-1}(s-M)_-\,ds \zeta^p\,dx\\
\leq& \int_{K_{\rho} \times\{0\}} \int_u^M s^{q-1}(s-M)_-\,ds \zeta^p\,dx +\gamma \iint_Q(u-M)_-^p|\nabla \zeta|^p\,dxdt\\
&+\int_{0}^{\delta M^{q+1-p}\rho^p}\!\!\int_{K_\rho}\int_{K_\rho}\max \{(u-M)_-(x, t),(u-M)_-(y,t)\}^p|\zeta(x,t)-\zeta(y,t)|^p\,d\mu dt \nonumber\\
&+\gamma \mathop{\mathrm{ess}\,\sup}_{\stackrel{0<t<\delta M^{q+1-p}\rho^p}{x\in \mathrm{supp}\,\zeta(\cdot,t)}}\int_{\mathbb{R}^N \backslash K_\rho} \frac{(u-M)_-^{p-1}(y,t)}{|x-y|^{N+sp}}\,dy\iint_Q(u-M)_-\zeta^p(x,t)\,dxdt\nonumber\\
\leq &\int_{K_{\rho} \times\{0\}} \int_u^M s^{q-1}(s-M)_-\,ds \zeta^p\,dx+\gamma\frac{\delta M^{q+1}}{\sigma^{N+p}}|K_\rho|,
\end{align*}
where $\gamma$ depends only on $N, p, s, q, \Lambda$. The rest of the proof is standard, and we refer the readers to Lemma 6.2 in \cite{BDG23} for details.
\end{proof}

Next, we introduce a measure shrinking lemma.

\begin{lemma}
\label{lem-6-4}
Let $0<p-1\leq q$, constants $\delta, \alpha\in(0,1)$ and $M>0$. Suppose that $u$ is a nonnegative, weak supersolution to \eqref{1.1}. Let $\rho \in(0,1)$ and $K_{2 \rho}(x_0) \times(t_0, t_0+\delta M^{q+1-p} \rho^p] \subset E_T$. For any $\nu\in(0,1)$, if 
\begin{align}
\label{6.6}
|[u(\cdot, t) \geq M] \cap K_{\rho}(x_0)| \geq \alpha|K_{\rho}| \quad \text { for all } t \in\left(t_0, t_0+\delta M^{q+1-p} \rho^p\right],
\end{align}
then there exists $\xi \in(0,1)$ depending on $N, p, s, q, \Lambda, \delta, \nu$ and $\alpha$ such that
\begin{align*}
|[u(\cdot, t) \leq \xi M] \cap K_{\rho}(x_o)| \leq \nu|K_{\rho}|
\end{align*}
for all time
\begin{align*}
t \in\left(t_0+\tfrac{1}{2} \delta M^{q+1-p} \rho^p, t_0+\delta M^{q+1-p} \rho^p\right].
\end{align*}
\end{lemma}

At this point, some preparations are needed. Let $(x_0,t_0)=(0,0)$. Set
\begin{align*}
I=(0,\delta M^{q+1-p} \rho^p], \quad \lambda I=\left((1-\lambda) \delta M^{q+1-p} \rho^p, \delta M^{q+1-p} \rho^p\right],
\end{align*}
and
\begin{align*}
Q=K_{2 \rho} \times I, \quad \lambda Q=K_{\lambda 2 \rho} \times \lambda I,
\end{align*}
where $\lambda \in(0,1)$. For some $c\in(0,1)$ that will be selected later, taking the sequence
\begin{align}
\label{6.7}
k_j:=c^j M, \ \ j=0,1,2\ldots.
\end{align}
Denote
\begin{align}
\label{6.8}
Y_j:=\sup_{t \in I} \mint_{K_{2\rho} \times\{t\}} \zeta^p \chi_{[u<k_j]}\,dx.
\end{align}
Here, the cutoff function $\zeta(x,t)=\zeta_1(x)\zeta_2(t)$ is piecewise smooth in $Q$, such that
\begin{align}
\label{6.9}
\left\{\begin{array}{c}
0 \leq \zeta \leq 1  \text { in } Q,  \quad \zeta=1  \text { in } \frac{1}{2} Q, \quad \zeta=0   \text { in } Q \backslash \frac{3}{4} Q, \\
|\nabla \zeta_1| \leq \frac{2}{\rho}, \quad 0 \leq \partial_t \zeta_2 \leq \frac{4}{\delta M^{q+1-p} \rho^p}, \\
\text { the sets }[x \in K_{2 \rho}: \zeta_1(x)>a] \text { are convex for all } a \in(0,1).
\end{array}\right.
\end{align}

In the next step, we give the estimate of $Y_j$ as a crucial tool to prove Lemma \ref{lem-6-4}.

\begin{lemma}
	\label{lem-6-5}
Suppose the conditions in Lemma \ref{lem-6-4} hold. Then there exist $\sigma, c \in(0,1)$ depending on $N,p,s,q,\Lambda,\delta,\alpha$ and $\nu$, such that for every $\nu\in(0,1)$ there holds either
\begin{align*}
Y_j \leq \nu
\end{align*}
or
\begin{align*}
Y_{j+1} \leq \max \left\{\nu, \sigma Y_j\right\} 
\end{align*}
with $j\in\mathbb{N}\cup\{0\}$.
\end{lemma}
\begin{proof}
We prove this lemma starts by showing an integral inequality under the assumption 
\begin{align}
\label{6.10}
\partial_t u^q \in C(I;L^1(K_{2 \rho})),
\end{align}
and drops it later. As a consequence of Proposition \ref{pro-3-2}, $u_k:=k-(k-u)_{+}$ with $k \in(0, M)$ is a nonnegative, weak supersolution to \eqref{1.1} in $Q$, reads that
\begin{align}
\label{6.11}
&\partial_t u_k^q-\operatorname{div} (|\nabla u_k|^{p-2}\nabla u_k)\nonumber\\
& \ \ \ \ \ \ + \mathrm{P.V.}\int_{\mathbb{R}^N}|u_k(x,t)-u_k(y,t)|^{p-2}(u_k(x,t)-u_k(y,t))\,dy\geq 0 \quad \text { weakly in } Q .
\end{align}
Testing \eqref{6.11} with the function
\begin{align}
\label{6.12}
\frac{\zeta^p}{\left[k-(k-u)_{+}+ck\right]^{p-1}},
\end{align}
where $\zeta$ is given in \eqref{6.9}, and constants $c\in(0,1),k\in(0,M)$ will be determined later. Then for a.e. $t\in I$, we have
\begin{align}
\label{6.13}
& \partial_t \int_{K_{2 \rho} \times\{t\}} \zeta^p \Phi_k(u)\,dx+\int_{K_{2\rho} \times\{t\}} \zeta^p\left|\nabla \Psi_k(u)\right|^p\,dx\nonumber\\
\leq&\int_{K_{2 \rho} \times\{t\}}\left|\nabla \Psi_k(u)\right|^{p-1} \zeta^{p-1}|\nabla \zeta|\,dx +\int_{K_{2 \rho} \times\{t\}} \Phi_k(u) \partial_t \zeta^p\,dx \nonumber\\
&+\int_{K_{2 \rho} \times\{t\}}\int_{K_{2\rho}}\frac{U_k(x,y,t)}{|x-y|^{N+sp}} \left(\frac{\zeta^p(x,t)}{[u_k(x,t)+ck]^{p-1}}-\frac{\zeta^p(y,t)}{[u_k(y,t)+ck]^{p-1}}\right)\,dxdy\nonumber\\
&+2\int_{K_{2 \rho} \times\{t\}}\int_{\mathbb{R}^N\backslash K_{2\rho}} \frac{\zeta^p(x,t)U_k(x,y,t)}{[u_k(x,t)+ck]^{p-1}|x-y|^{N+sp}}\,dxdy\nonumber\\
=:&J_1+J_2+J_3+J_4,
\end{align}
where
\begin{align*}
&U_k(x,y,t):=|u_k(x,t)-u_k(y,t)|^{p-2}(u_k(x,t)-u_k(y,t)),\\
& \Phi_k(u):=\int_0^{(k-u)_{+}} \frac{q(k-s)^{q-1}}{(k-s+c k)^{p-1}}\,ds, \\
& \Psi_k(u):=\ln \left(\frac{k(1+c)}{k(1+c)-(k-u)_{+}}\right).
\end{align*}
The existence of the term containing time derivative is ensured by assumption \eqref{6.10}. 

Next, we estimate $J_1$--$J_4$ in \eqref{6.13}, separately. 

\textbf{Estimate of $J_1$:} By Young's inequality, it holds that
\begin{align*}
J_1\leq  \frac{1}{2}\int_{K_{2 \rho} \times\{t\}}\left|\nabla \Psi_k(u)\right|^p \zeta^p\,dx+\gamma(p) \int_{K_{2 \rho}}|\nabla \zeta|^p\,dx.
\end{align*}

\textbf{Estimate of $J_2$:} Since $0<p-1\leq q$, we compute 
\begin{align}
\label{6.14}
\Phi_k(u) &\leq \int_0^k \frac{q(k-s)^{q-1}}{(k-s+c k)^{p-1}}\,ds\nonumber\\
&=k^{q+1-p} \int_0^1 \frac{q s^{q-1}}{(s+c)^{p-1}}\,ds\nonumber\\
&\leq \gamma(p,q) k^{q+1-p} \ln \left(\frac{1+c}{c}\right).
\end{align}
By using \eqref{6.14}, \eqref{6.9} and $k\in(0, M)$, it follows that
\begin{align*}
J_2 \leq \gamma \frac{k^{q+1-p}}{\delta M^{q+1-p} \rho^p} \ln \left(\frac{1+c}{c}\right)|K_{2 \rho}| \leq \frac{4 \gamma(p,q)}{\delta \rho^p} \ln \left(\frac{1+c}{c}\right)|K_{2 \rho}| .
\end{align*}

\textbf{Estimate of $J_3$:} $J_3$ can be estimated in virtue of \eqref{6.9} as
\begin{align*}
J_3\leq&\int_{K_{2 \rho} \times\{t\}}\int_{K_{2\rho}}\frac{|u_k(x,t)-u_k(y,t)|^{p-2}(u_k(x,t)-u_k(y,t))}{|x-y|^{N+sp}}\times \frac{\zeta^p(x,t)-\zeta^p(y,t)}{[u_k(x,t)+ck]^{p-1}}\,dxdy\\
\leq&\gamma(N,p,s,q,\Lambda)\frac{1}{\rho^p}|K_{2\rho}|,
\end{align*}
where we drop the non-positive term
\begin{align*}
\int_{K_{2 \rho} \times\{t\}}\int_{K_{2\rho}}\frac{U_k(x,y,t)}{|x-y|^{N+sp}}\left(\frac{\zeta^p(y,t)}{[u_k(x,t)+ck]^{p-1}}-\frac{\zeta^p(y,t)}{[u_k(y,t)+ck]^{p-1}}\right)\,dxdy.
\end{align*}

\textbf{Estimate of $J_4$:}
Note that 
\begin{align*}
\frac{|y|}{|y-x|}\leq1+\frac{|x|}{|y-x|}\leq 1+\frac{3/2\rho}{1/2\rho}\leq 4
\end{align*}
for any $|x|\leq \frac{3}{2}\rho$  and $|y|\geq 2\rho$. Thus we evaluate
\begin{align*}
J_4&\leq\gamma\int_{\mathbb{R}^N\backslash K_{2\rho}}\frac{1}{|y|^{N+sp}}\,dy\int_{K_{2 \rho} \times\{t\}}\frac{u_k(x,t)^{p-1}\zeta^p(x,t)}{[u_k(x,t)+ck]^{p-1}}\,dx\\
&\leq\gamma(N,p,s,q,\Lambda)\frac{1}{\rho^p}|K_{2\rho}|.
\end{align*}
Inserting estimates $J_1$--$J_4$ into \eqref{6.13}, we derive for a.e. $t\in I$ that
\begin{align}
\label{6.15}
\partial_t \int_{K_{2 \rho} \times\{t\}} \zeta^p \Phi_k(u)\,dx+\int_{K_{2 \rho} \times\{t\}} \zeta^p\left|\nabla \Psi_k(u)\right|^p\,dx \leq \frac{\gamma}{\delta \rho^p} \ln \left(\frac{1+c}{c}\right)\left|K_{2 \rho}\right|
\end{align}
with $\gamma$ depending only on $N, p, s, q, \Lambda$. Here, we can chose $c\in(0,\frac{1}{3})$ such that $\ln \left(\frac{1+c}{c}\right) \geq 1$. From the measure information \eqref{6.6} and $k<M$, we have
\begin{align*}
\left|[\Psi_k(u)=0 ]\cap K_{\rho}\right| \geq \alpha 2^{-N}|K_{2 \rho}| \quad \text { for all } t \in I .
\end{align*}
Employing Poincar\'{e} inequality in Lemma \ref{lem-2-6}, it gives that
\begin{align}
\label{6.16}
\int_{K_{2 \rho} \times\{t\}} \zeta^p \Psi_k^p(u)\,dx \leq \frac{\gamma_* \rho^p}{\alpha^p} \int_{K_{2 \rho} \times\{t\}} \zeta^p\left|\nabla \Psi_k(u)\right|^p\,dx \quad \text { for a.e. } t \in I,
\end{align}
where $\gamma_*$ is the Sobolev constant depending only on $p,N$. Taking \eqref{6.15} and \eqref{6.16} into account, we arrive at
\begin{align}
\label{6.17}
\partial_t \int_{K_{2 \rho} \times\{t\}} \zeta^p \Phi_k(u)\,dx +\frac{\alpha^p}{\gamma_* \rho^p} \int_{K_{2 \rho} \times\{t\}} \zeta^p \Psi_k^p(u)\,dx \leq \frac{\gamma}{\delta \rho^p} \ln \left(\frac{1+c}{c}\right)|K_{2 \rho}|,
\end{align}
which is the desired integral inequality.

Since the integral inequality \eqref{6.17} has been obtained, we can proceed as in  \cite[\S 6.3.1.2-\S 6.3.1.4]{BDG23} to prove there holds $Y_{j+1}\leq \max \left\{\nu, \sigma Y_j\right\}$ under the assumption \eqref{6.10}. Then, we can remove \eqref{6.10} by following the argument in  \cite[\S 6.3.1.5]{BDG23}. Moreover, the convergence of the nonlocal term can be verified by Lemma \ref{lem-2-8} $(\rm{{\romannumeral 1}})$. We omit the proof for short.
\end{proof}
\begin{proof}[\textbf{Proof of Lemma \ref{lem-6-4}}]
Iterating Lemma \ref{lem-6-3} yields that
\begin{align*}
Y_{j_0} \leq \max \left\{\nu, \sigma^{j_0} Y_0\right\} \quad \text{for}\ j_0\in\mathbb{N}.
\end{align*}
From the definition of $Y_j$ in \eqref{6.8}, we know $Y_0\leq 1$. Now, choosing $j_0$ such that $\sigma^{j_0} \leq \nu$, further to obtain $Y_{j_0} \leq \nu$. In view of \eqref{6.7}, \eqref{6.8} and \eqref{6.9}, we have
\begin{align*}
\frac{1}{|K_{2 \rho}|}\left|[u(\cdot, t)\leq c^{j_0} M]  \cap K_{\rho}\right| \leq Y_{j_0} \leq \nu \quad \text { for all } t \in \frac{1}{2} I .
\end{align*}
Let $\xi=c^{j_0}$, we conclude the proof by replacing $\nu$ by $2^{-N} \nu$ and adjusting some constants.
\end{proof}

We now proceeding in proving Proposition \ref{pro-6-1}.
\begin{proof}[\textbf{Proof of Proposition \ref{pro-6-1}}] The measure theoretical hypothesis \eqref{6.1} implies that
\begin{align*}
\left|[u(\cdot,t_0)\geq M] \cap K_{4 \rho}(x_0)\right| \geq 4^{-N} \alpha\left|K_{4 \rho}\right|.
\end{align*}
This along with Lemma \ref{lem-6-3} gives there exist $\delta$ and $\varepsilon$ in $(0,1)$ depending only on $N,p,s,q,\Lambda,\alpha$ such that
\begin{align}
\label{6.20}
\left|[u(\cdot,t)\geq \varepsilon M] \cap K_{4 \rho}(x_0)\right| \geq \frac{1}{2} 4^{-N} \alpha|K_{4 \rho}|
\end{align}
for any
\begin{align*}
t \in(t_0,t_0+\delta M^{q+1-p}(4 \rho)^p].
\end{align*}
With the help of \eqref{6.20}, we can employ Lemma \ref{lem-6-4} to derive that for given $\nu \in(0,1)$, there exists $\xi \in(0,1)$ depending only on $N, p, s, q, \Lambda, \nu, \delta, \alpha$ such that
\begin{align*}
\left|[u(\cdot,t)\leq \xi\varepsilon M] \cap K_{4 \rho}(x_0)\right|
\leq \nu\left|K_{4 \rho}\right|
\end{align*}
for
\begin{align*}
t \in\left(t_0+\tfrac{1}{2} \delta(\varepsilon M)^{q+1-p}(4 \rho)^p, t_0+\delta(\varepsilon M)^{q+1-p}(4 \rho)^p\right].
\end{align*}
Notice for every time
\begin{align*}
\bar{t} \in\left(t_0+\tfrac{3}{4} \delta(\varepsilon M)^{q+1-p}(4 \rho)^p, t_0+\delta(\varepsilon M)^{q+1-p}(4 \rho)^p\right],
\end{align*}
we have
\begin{align*}
\left(x_0, \bar{t}\right)+Q_{4 \rho}(\theta)\subset K_{4 \rho}(x_0) \times\left(t_0+\tfrac{1}{2} \delta(\varepsilon M)^{q+1-p}(4 \rho)^p, t_0+\delta(\varepsilon M)^{q+1-p}(4 \rho)^p\right]
\end{align*}
with $\theta=\frac{1}{4} \delta(\xi \varepsilon M)^{q-1+p}$. Thus, we deduce for any $\bar{t}$ that
\begin{align}
\label{6.21}
\left|[u \leq \xi \varepsilon M]\cap(x_0, \bar{t})+Q_{4 \rho}(\theta)\right| \leq \nu|Q_{4 \rho}(\theta)|.
\end{align}
Thanks to the measure information \eqref{6.21}, it permits us to exploit the De Giorgi-type Lemma \ref{lem-6-2} in $(x_0, \bar{t})+Q_{4 \rho}(\theta))$ with $M$ replaced by $\xi\varepsilon M$, we also choose $\nu$ from Lemma \ref{lem-6-2}, then we deduce that
\begin{align*}
u \geq \frac{1}{2} \xi \varepsilon M \quad \text{ a.e. in}  (x_0, \bar{t})+Q_{2 \rho}(\theta).
\end{align*}
Due to the arbitrariness of $\bar{t}$, we get the desired results.
\end{proof}

\section{Harnack inequality}
\label{sec7}
In this section, we devote to establishing the Harnack inequality in Theorem \ref{thm-1-6}. The first step is going to scale functions.

\subsection{Scaling functions} Change the variables
\begin{align*}
z \rightarrow \frac{x-x_0}{\rho}, \quad z^{\prime} \rightarrow \frac{y-x_0}{\rho}, \quad \tau \rightarrow [u(x_0, t_0)]^{p-q-1} \frac{t-t_0}{\rho^p} .
\end{align*}
Consider the following rescaled function
\begin{align}
\label{7.1}
v(z,\tau):=\frac{u\left(x_0+\rho z, t_0+[u(x_0, t_0)]^{q+1-p} \rho^p \tau\right)}{u(x_0,t_0)} \quad \text { in } \widetilde{Q}_8=K_8 \times(-8^p, 8^p).
\end{align}
Via formula calculation, we can check that $v(0,0)=1$ and $v$ is a bounded continuous nonnegative solution to
\begin{align}
\label{7.2}
&\partial_\tau v^q-\mathrm{div}(|\nabla v|^{p-2} \nabla v)\nonumber\\
&\ \ \ \ \ +\mathrm{P.V.} \int_{\mathbb{R}^N} \widetilde{K}(z,z',\tau)|v(z,\tau)-v(z',\tau)|^{p-2}(v(z,\tau)-v(z',\tau))\,dz'=0 \quad \text { in } \widetilde{Q}_8
\end{align}
with 
\begin{align*}
\widetilde{K}(z,z',\tau)=\rho^{N+p}K(x_0+\rho z,x_0+\rho z', t_0+[u(x_0,t_0)]^{q+1-p}\rho^p \tau),
\end{align*}
satisfying
\begin{align*}
\frac{\rho^{p-sp}\Lambda^{-1}}{|z-z'|^{N+sp}}\leq\widetilde{K}(z,z',\tau)\leq \frac{\rho^{p-sp}\Lambda}{|z-z'|^{N+sp}}.
\end{align*}

Now, our intention turns to demonstrate there exist $\gamma>1$ and $\sigma \in(0,1)$ depending only on $N,p,s,q,\Lambda$ and $\|u\|_{L^\infty(\mathbb{R}^N\times(0,T))}$ such that
\begin{align}
\label{7.3}
v \geq \gamma^{-1} \quad \text { in } K_1 \times(-\sigma, \sigma).
\end{align}
Notice the Harnack inequality on the left-hand side is a straightforward consequence of \eqref{7.3}. 

\subsection{The supremum of function $v$ in $K_1$.} For $\tau\in(0,1)$, define
\begin{align*}
M_\tau:=\sup_{K_\tau} v(\cdot,0), \quad N_\tau:=(1-\tau)^{-\beta}, \quad \text { where } \beta=\frac{p}{q+1-p} .
\end{align*}
Observe that $M_0=N_0=1$, $N_\tau\rightarrow\infty$ as $\tau\rightarrow 1$, and $M_\tau$ is bounded due to $v$ is a bounded function. Therefore, due to $M_\tau=N_\tau$ must have roots, we denote the largest one by $\tau_*$, specifically,
\begin{align*}
M_{\tau_*}=N_{\tau_*} \quad \text { and } \quad M_\tau \leq N_\tau \quad \text { for all } \tau \geq \tau_*.
\end{align*}
Pick $\bar{\tau}\in\left(\tau_*,1\right)$ such that
\begin{align*}
N_{\bar{\tau}}=(1-\bar{\tau})^{-\beta}=4(1-\tau_*)^{-\beta}, \quad \text { i.e., } \quad \bar{\tau}=1-4^{-\frac{1}{\beta}}\left(1-\tau_*\right),
\end{align*}
and let
\begin{align*}
2 r:=\bar{\tau}-\tau_*=(1-4^{-\frac{1}{\beta}})(1-\tau_*).
\end{align*}
Clearly, $M_{\tau_*}$ can be achieved at some $\bar{x} \in K_{\tau_*}$ because $v$ is a continuous function. Hence, there holds $K_{2r}(\bar{x}) \subset K_{\bar{\tau}}, M_{\bar{\tau}} \leq N_{\bar{\tau}}$, and
\begin{align}
\label{7.4}
\sup_{K_{\tau_*}} v(\cdot,0) & =M_{\tau_*}=v(\bar{x}, 0) \leq \sup _{K_{2 r}(\bar{x})} v(\cdot,0) \nonumber\\
& \leq \sup_{K_{\bar{\tau}}} v(\cdot, 0)=M_{\bar{\tau}} \leq N_{\bar{\tau}}=4(1-\tau_*)^{-\beta}.
\end{align}
\subsection{Expanding of positivity of $v$.}
Set the cylinder
\begin{align*}
\widetilde{Q}_r(\theta_*):=K_r(\bar{x}) \times(-\theta_* r^p, \theta_* r^p), \quad \theta_*:=(1-\tau_*)^{-\beta(q+1-p)}
\end{align*}
with center $(\bar{x}, 0)$.
By the choices of $\beta$ and $r$, we compute
\begin{align}
\label{7.5}
\theta_* r^p=(1-\tau_*)^{-p} r^p=2^{-p}\left(1-4^{-\frac{1}{\beta}}\right)^p=:c,
\end{align}
which implies that $$\widetilde{Q}_{2 r}\left(\theta_*\right)=K_{2 r}(\bar{x}) \times\left(-2^p c, 2^p c\right). $$
Obviously,
\begin{align}
\label{7.6}
r=c^{\frac{1}{p}}\left(1-\tau_*\right).
\end{align}

We next discuss the estimate of the supremum of $v$ in $\widetilde{Q}_r\left(\theta_*\right)$.

\begin{lemma}
\label{lem-7-1}
Let $v$ be defined as in \eqref{7.1} and $0<p-1<q<\min\left\{p^2-1,\frac{N(p-1)}{(N-p)_{+}}\right\}$. There exists a constant $\gamma$ depending only on $N, p, s, q, \Lambda,\|u\|_{L^\infty(\mathbb{R}^N\times(0, T))}$ such that 
\begin{align*}
\sup_{\widetilde{Q}_r\left(\theta_*\right)} v \leq \gamma(1-\tau_*)^{-\beta} .
\end{align*}
\end{lemma}
\begin{proof}
An application of Theorem \ref{thm-1-5} to $v$ over the cylinder  $\widetilde{Q}_r\left(\theta_*\right) \subset \widetilde{Q}_{2 r}\left(\theta_*\right)$ leads to 
\begin{align}
\label{7.7}
\sup _{\widetilde{Q}_r(\theta_*)} v^q  \leq& \frac{\gamma}{c^{\frac{N}{\lambda}}}\left(\int_{K_{2r}(\bar{x})} v^q(x,0)\,dx\right)^{\frac{p}{\lambda}}+\gamma\left(\frac{c}{r^p}\right)^{\frac{q}{q+1-p}}\nonumber\\
&+\gamma\frac{c}{r^p}\left[\mathrm{Tail}_\infty(u;\bar{x},r;-2^p c,2^p c)\right]^{p-1}
\!+\!\gamma \left(\frac{c}{r^p}\right)^{\frac{p-N}{\lambda}}[\mathrm{Tail}_\infty(u;\bar{x},r;-2^p c,2^p c)]^{\frac{p(p-1)}{\lambda}} \nonumber\\
\leq&\frac{\gamma}{c^{\frac{N}{\lambda}}}\left(\int_{K_{2r}(\bar{x})} v^q(x,0)\,dx\right)^{\frac{p}{\lambda}}+\gamma\left(\frac{c}{r^p}\right)^{\frac{q}{q+1-p}}+\gamma\frac{c}{r^p}+\gamma \left(\frac{c}{r^p}\right)^{\frac{p-N}{\lambda}}\nonumber\\
\leq&\gamma\left[(1-\tau_*)^{-\frac{pq}{q+1-p}}+(1-\tau_*)^{-p}+(1-\tau_*)^\frac{-p(p-N)}{\lambda}\right],
\end{align}
where $\gamma$ depends only on $N,p,s,q,\Lambda,\|u\|_{L^\infty(\mathbb{R}^N\times(0,T))}$. In the last step of \eqref{7.7}, we used \eqref{7.4} and \eqref{7.6}.

By the definitions of $\beta, \lambda$ and the assumption $0<p-1<q$, we compute
\begin{align}
\label{7.8}
-p\geq\frac{-pq}{q+1-p}=-\beta q
\end{align}
and
\begin{align}
\label{7.9}
\frac{-p(p-N)}{\lambda}=\frac{pq(p-N)}{N(q+1-p)-pq}\geq\frac{-pqN}{N(q+1-p)}=-\beta q.
\end{align}
Since $\tau_*\in(0,1)$, the claim follows from \eqref{7.7}--\eqref{7.9}.
\end{proof}

With Lemma \ref{lem-7-1}, we present the following result which provides the condition of the expansion of positivity.

\begin{lemma}
\label{lem-7-2}
Let $v$ be defined as in \eqref{7.1}. Let $p>N$ and $0<p-1<q<p^2-1$. There exist constants $\delta, \bar{c}, \alpha\in(0,1)$ just depending on $N,p,s,q,\Lambda,\|u\|_{L^\infty(\mathbb{R}^N\times(0,T))}$ such that 
\begin{align*}
\left|\left[v(\cdot, t) \geq \bar{c}(1-\tau_*)^{-\beta}\right] \cap K_r(\bar{x})\right| \geq \alpha|K_r| \quad \text { for all } t\in\left[-\delta \theta_*r^p, \delta\theta_*r^p\right].
\end{align*}
\end{lemma}
\begin{proof}
By using Theorem \ref{thm-1-5} to function $v$ over the cylinder $\widetilde{Q}_{\frac{1}{2} r}(\delta \theta_*) \subset \widetilde{Q}_r(\delta \theta_*)$ yields
\begin{align*}
(1-\tau_*)^{-\beta q}  =&v^q(\bar{x}, 0) \leq \sup_{K_{\frac{1}{2} r}(\bar{x})} v^q(\cdot,0) \\
 \leq& \frac{\gamma}{\left(\delta \theta_* r^p\right)^{\frac{N}{\lambda}}}\left(\int_{K_r(\bar{x})} v^q(x,t)\,dx \right)^{\frac{p}{\lambda}}+\gamma(\delta \theta_*)^{\frac{q}{q+1-p}}+\gamma\delta \theta_*+\gamma(\delta \theta_*)^{\frac{p-N}{\lambda}} \\
 \leq&\frac{\gamma}{\left(\delta \theta_* r^p\right)^{\frac{N}{\lambda}}}\left(\int_{K_r(\bar{x})} v^q(x,t)\,dx\right)^{\frac{p}{\lambda}}+\gamma \delta^{\frac{q}{q+1-p}}\left(1-\tau_*\right)^{-\beta q}\\
&+\gamma\delta(1-\tau_*)^{-\beta(q+1-p)}+\gamma \delta^{\frac{p-N}{\lambda}}(1-\tau_*)^{-\beta(q+1-p)\frac{p-N}{\lambda}}\\
\leq&\frac{\gamma}{\left(\delta \theta_* r^p\right)^{\frac{N}{\lambda}}}\left(\int_{K_r(\bar{x})} v^q(x,t)\,dx\right)^{\frac{p}{\lambda}}+\gamma\delta^{\frac{p-N}{\lambda}}(1-\tau_*)^{-\beta q}
\end{align*}
for all $t\in\left[-\delta\theta_*r^p,\delta\theta_*r^p\right]$. In the above display, we also employed the definition of $\theta_*$ and the fact $0<(p-N)/\lambda<1$. With taking $\delta$ such that $\delta^{\frac{p-N}{\lambda}}\leq\frac{1}{2}$, we have
\begin{align*}
(1-\tau_*)^{-\beta q}\leq \gamma\left(\int_{K_r(\bar{x})} v^q(x,t)\,dx\right)^{\frac{p}{\lambda}},
\end{align*}
where $\gamma$ depends only on $N, p, s, q, \Lambda$ and $\|u\|_{L^\infty(\mathbb{R}^N\times(0,T))}$, since $\theta_* r^p=c$ depends only on $p$ and $q$.
In addition, we deduce from Lemma \ref{lem-7-1} that
\begin{align*}
 \int_{K_r(\bar{x})} v^q(x,t)\,dx=&\int_{K_r(\bar{x})} v^q(x, t) \chi_{\left[v<\bar{c}(1-\tau_*)^{-\beta}\right]}\,dx+\int_{K_r(\bar{x})} v^q(x,t) \chi_{\left[v \geq \bar{c}(1-\tau_*)^{-\beta}\right]}\,dx \\
\leq&  \bar{c}^q(1-\tau_*)^{-\beta q}(2r)^N+\gamma^q(1-\tau_*)^{-\beta q}|[v \geq \bar{c}\left(1-\tau_*\right)^{-\beta}] \cap K_r(\bar{x})|
\end{align*}
with $\bar{c}\in(0,1)$ to be specified. In light of the convexity of $s^{\frac{p}{\lambda}}$, we can derive
\begin{align*}
(1-\tau_*)^{-\beta q} \leq & \gamma \left[\bar{c}^q(1-\tau_*)^{-\beta q}(2 r)^N\right]^{\frac{p}{\lambda}}+\gamma  (1-\tau_*)^{-\beta q \frac{p}{\lambda}}\left|\left[v \geq \bar{c}(1-\tau_*)^{-\beta}\right] \cap K_r(\bar{x})\right|^{\frac{p}{\lambda}} \\
\leq & \gamma \bar{c}^{q \frac{p}{\lambda}}(1-\tau_*)^{-\beta q}+\gamma  (1-\tau_*)^{-\beta q \frac{p}{\lambda}}\left|\left[v \geq \bar{c}(1-\tau_*)^{-\beta}\right] \cap K_r(\bar{x})\right|^{\frac{p}{\lambda}}.
\end{align*}
Fixing $\bar{c}$ such that $\gamma\bar{c}^{q\frac{p}{\lambda}}=\frac{1}{2}$, we obtain
\begin{align*}
(1-\tau_*)^{-\beta q} \leq \gamma(1-\tau_*)^{-\beta q \frac{p}{\lambda}}\left|\left[v \geq \bar{c}(1-\tau_*)^{-\beta}\right] \cap K_r(\bar{x})\right|^{\frac{p}{\lambda}}
\end{align*}
with $\gamma$ depending only on $N,p,s,q,\Lambda$ and $\|u\|_{L^\infty(\mathbb{R}^N\times(0,T))}$. Thus, we complete the proof based on the definitions of $r,\beta$ and $\lambda$.
\end{proof}

The next pointwise estimate is a direct consequence of Proposition \ref{pro-6-1} and Lemma \ref{lem-7-2}.

\begin{lemma}
\label{lem-7-3}
Let $v$ be defined as in \eqref{7.1}. Let $p>N$ and $0<p-1<q<p^2-1$. There exist constants $\eta, \delta \in(0,1)$ only depending on $N, p, s, q, \Lambda, \|u\|_{L^\infty(\mathbb{R}^N\times(0, T))}$, such that
\begin{align*}
v \geq \eta(1-\tau_*)^{-\beta} \quad \text { in } K_{2 r}(\bar{x}) \times\left[-\tfrac{1}{2} \delta \theta_* r^p, \delta \theta_* r^p\right].
\end{align*} 
\end{lemma}

In the next step, we expand the pointwise positivity of $v$ defined as in \eqref{7.1} to  $K_2(\bar{x})$, and thus justify \eqref{7.3}. This proof proceeds by a comparison argument.
\subsection{A comparison argument}
We give the following initial-boundary value problem
\begin{align}
\label{7.10}
\left\{\begin{array}{cl}
\partial_t w^q-\mathrm{div}(|\nabla w|^{p-2} \nabla w)+\mathcal{L} w=0 & \text { in } K_4(\bar{x}) \times(-\sigma, 1], \\
w=0 & \text { in } \mathbb{R}^N\backslash K_4(\bar{x}) \times(-\sigma, 1], \\
w^q(\cdot,-\sigma)=\eta(1-\tau_*)^{-N} \chi_{K_{2 r}(\bar{x})}(\cdot) & \text { in } K_4(\bar{x}),
\end{array}\right.
\end{align}
where $\mathcal{L}$ is given by \eqref{1.2}, $\eta$ is defined as in Lemma \ref{7.3}, $\sigma \in\left(0, \frac{1}{2} \delta c\right)$ to be determined and $c$ is defined as in \eqref{7.5}.

\begin{lemma}
\label{lem-8-4}
Let $0<p-1<q$ and $p>N$. Let $v$ defined as in \eqref{7.1} be a nonnegative weak solution to \eqref{7.2} and $w$ be a nonnegative weak solution to \eqref{7.10}. Then we have $w \leq v$ a.e. in $K_4(\bar{x}) \times[-\sigma, 1]$.
\end{lemma}

\begin{proof}
Since $v$ is a nonnegative solution to \eqref{7.2}, we know that $w \leq v$ in $\mathbb{R}^N\backslash K_4(\bar{x}) \times(-\sigma, 1]$. Thanks to the assumptions on $p,q$, we have $\beta q>N$, this combines with Lemma \ref{lem-7-3} gives $w(\cdot,-\sigma)\leq v(\cdot,-\sigma)$ a.e. in $K_4(\bar{x})$, and thus we derive from Proposition \ref{pro-3-4} that  $w \leq v$ a.e. in $K_4(\bar{x}) \times[-\sigma, 1]$.
\end{proof}

Finally, we end this section with proving Theorem \ref{thm-1-6}.

\begin{proof}[\textbf{Proof of Theorem \ref{thm-1-6}}]
In view of Lemma \ref{lem-8-4}, we only need to obtain the pointwise positivity of $w$ as below,
\begin{align}
\label{7.11}
w \geq \gamma^{-1} \quad \text { in } K_2(\bar{x}) \times\left[-\tfrac{1}{4} \sigma, \tfrac{1}{4}\sigma\right]
\end{align}
with $\gamma>1$ and $\sigma \in\left(0, \frac{1}{2} \delta c\right)$ depending only on $N,p,s,q,\Lambda,\|u\|_{L^\infty(\mathbb{R}^N\times(0,T))}$. In order to get \eqref{7.11}, we utilize Proposition \ref{pro-5-1} over the cylinder $K_2(\bar{x}) \times[-\sigma, \sigma]$. Then for all $t \in[-\sigma, \sigma]$, there holds
\begin{align}
\label{7.12}
\int_{K_1(\bar{x})} w^q(x,-\sigma)\,dx &\leq \gamma \int_{K_2(\bar{x})} w^q(x,t)\,dx +\gamma \sigma^{\frac{q}{q+1-p}} +\gamma\sigma\left[\mathrm{Tail}_\infty(w;\bar{x},1;-\sigma,\sigma)\right]^{p-1}\nonumber\\
&\leq  \gamma \int_{K_2(\bar{x})} w^q(x,t)\,dx+\gamma\sigma,
\end{align}
where we used $\sigma\in(0,1)$ and $p>1$. Then, by virtue of the initial data in \eqref{7.10}, we evaluate the left-hand side of \eqref{7.12} as
\begin{align}
\label{7.13}
\int_{K_1(\bar{x})} w^q(x,-\sigma)\,dx=\eta(1-\tau_*)^{-N}(4r)^N=2^N \eta\left(1-4^{-\frac{1}{\beta}}\right)^N=:\eta c_0 .
\end{align}
Now, we pick $\sigma$ such that $\gamma\sigma=\frac{1}{2}\eta c_0$, a combination of \eqref{7.12} and \eqref{7.13} leads to
\begin{align*}
\frac{1}{2} \eta c_0 \leq \gamma \int_{K_2(\bar{x})} w^q(x, t)\,dx  \quad \text { for all } t \in[-\sigma, \sigma].
\end{align*}
Applying Theorem \ref{thm-1-5} in $K_2(\bar{x}) \times\left[-\frac{1}{4} \sigma, \frac{1}{4} \sigma\right]\subset K_4(\bar{x}) \times\left[-\sigma, \frac{1}{2} \sigma\right]$ along with the initial data in \eqref{7.10} gives 
\begin{align} 
\label{7.14}
\sup _{K_2(\bar{x}) \times\left[-\frac{1}{4} \sigma, \frac{1}{4} \sigma\right]} w^q &\leq \frac{\gamma}{\sigma^{\frac{N}{\lambda}}}\left(\eta c_0\right)^{\frac{p}{\lambda}}+\gamma \sigma^{\frac{q}{q+1-p}}+\gamma\sigma+\gamma\sigma^{\frac{p-N}{\lambda}}\nonumber\\
&\leq \frac{\gamma}{\sigma^{\frac{N}{\lambda}}}\left(\eta c_0\right)^{\frac{p}{\lambda}}+\gamma\sigma^{\frac{p-N}{\lambda}}=\gamma \sigma\times\sigma^{\frac{p-N}{\lambda}-1}\leq \gamma_1\sigma,
\end{align}
where $\gamma_1$ depends only on $N,p,s,q,\Lambda,\|u\|_{L^\infty(\mathbb{R}^N\times(0,T))}$ due to the choice of $\sigma$. In light of \eqref{7.14}, we derive for all $t \in\left[-\frac{1}{4} \sigma, \frac{1}{4} \sigma\right]$ that
\begin{align}
\label{7.16}
\frac{1}{2} \eta c_0 & \leq \gamma \int_{K_2(\bar{x})} w^q(x,t)\,dx\nonumber\\
&=\gamma \int_{K_2(\bar{x}) \cap[w<b]} w^q(x,t)\,dx +\gamma \int_{K_2(\bar{x}) \cap[w \geq b]} w^q(x,t)\,dx\nonumber\\
& \leq \gamma b^q|K_2|+\gamma \gamma_1 \eta c_0\big|[w(\cdot, t) \geq b]\cap K_2(\bar{x})\big|
\end{align}
with some $b>0$. In fact, we can choose $b$ such that
$\gamma b^q|K_2|=\tfrac{1}{4} \eta c_0$,
which together with \eqref{7.16} indicates 
\begin{align*}
|[w(\cdot, t) \geq b] \cap K_2(\bar{x})| \geq \frac{1}{4\gamma_2} \quad \text { for all } t \in\left[-\tfrac{1}{4} \sigma, \tfrac{1}{4} \sigma\right]
\end{align*}
with $\gamma_2=\gamma\gamma_1>1$ only depending on $N,p,s,q,\Lambda$ and $\|u\|_{L^\infty(\mathbb{R}^N\times(0,T))}$. With the above measure information at hand, the pointwise estimate of the function $w$ follows from Proposition \ref{pro-6-1}, consequently, the estimate of $v$. Finally, we establish the left-hand side Harnack inequality in Theorem \ref{thm-1-6} by the definition of $v$, and by choosing $\gamma,\sigma$ properly.

The argument to obtain Harnack inequality on the right-hand side is standard by using the continuity of weak solutions together with the Harnack inequality on the left-hand side, we refer the readers to \cite[pages $85$--$86$]{BDG23}. Up to now, we have finished the proof.
\end{proof}

\subsection*{Conflict of interest} 
The authors declare that there is no conflict of interest. 

\subsection*{Data availability}
No data was used for the research described in the article.

\subsection*{Acknowledgments}
This work was supported by the National Natural Science Foundation of China (No. 12071098) and the Fundamental Research Funds for the Central Universities (No. 2022FRFK060022). The paper was done when the second author visited Department of Mathematics, University of Craiova.  She would like to thank its hospitality during her stay and thank the China Scholarship Council (No. 202306120202). The research of V.D.~R\u adulescu was supported by the grant ``Nonlinear Differential Systems in Applied Sciences" of the Romanian Ministry of Research, Innovation and Digitization, within PNRR-III-C9-2022-I8/22.


\begin{thebibliography}{[a]}
	
\bibitem{AF89} E. Acerbi and N. Fusco, Regularity for minimizers of nonquadratic functionals: the case $1<p<2$, J.~Math. Anal. Appl. 140 (1) (1989) 115--135.

\bibitem{APT23} K. Adimurthi, H. Prasad and V. Tewary, Gradient regularity for mixed local-nonlocal quasilinear parabolic equations, arXiv:2307.02363v1.

\bibitem{ASD08} R. Alonso, M. Santillana and C. Dawson, On the diffusive wave approximation of the shallow water equations, European J. Appl. Math. 19 (5) (2008) 575--606. 

\bibitem{BGK23} A. Banerjee, P. Garain and J. Kinnunen, Lower semicontinuity and pointwise behavior of supersolutions for some doubly nonlinear nonlocal parabolic $p$-Laplace equations, Commun. Contemp. Math. 25 (8) (2023) 23pp.

\bibitem{BGK22} A. Banerjee, P. Garain and J. Kinnunen, Some local properties of subsolutions and supersolutions for a doubly nonlinear nonlocal $p$-Laplace equation, Ann. Mat. Pura Appl. 201 (4) (2022) 1717--1751.

\bibitem{BDG23} V. B\"{o}gelein, F. Duzaar, U. Gianazza, N. Liao and C. Scheven, H\"{o}lder continuity of the gradient of solutions to doubly non-linear parabolic equations, arXiv:2305.08539v1.

\bibitem{BDKS20} V. B\"{o}gelein, F. Duzaar, J. Kinnunen and C. Scheven,  Higher integrability for doubly nonlinear parabolic systems, J. Math. Pures Appl. 143 (9) (2020) 31--72.

\bibitem{BDMS18} V. B\"{o}gelein, F. Duzaar, P. Marcellini and C. Scheven, Doubly nonlinear equations of porous medium type, Arch. Ration. Mech. Anal. 229 (2) (2018) 503--545.

\bibitem{BDL21}V. B\"{o}gelein, F. Duzaar and N. Liao, On the H\"{o}lder regularity of signed solutions to a doubly nonlinear equation, J. Funct. Anal. 281 (9) (2021) 58pp.

\bibitem{BDM13} V. B\"{o}gelein, F. Duzaar and P. Marcellini, Parabolic systems with $p,q$-growth: a variational approach, Arch. Ration. Mech. Anal. 210 (1) (2013) 219--267.

\bibitem{BHS21} V. B\"{o}gelein, A. Heran, L. Sch\"{a}tzler and T. Singer, Harnack's inequality for doubly nonlinear equations of slow diffusion type, Calc. Var. Partial Differential Equations, 60 (6) (2021) 35pp.

\bibitem{BLS21} L. Brasco, E. Lindgren and M. Str\"{o}mqvist, Continuity of solutions to a nonlinear fractional diffusion equation, J. Evol. Equ. 21 (4) (2021) 4319--4381.

\bibitem{D93} E. DiBenedetto, Degenerate parabolic equations, Springer-Verlag, New York, 1993.	

\bibitem{DZZ21} M. Ding, C. Zhang and S. Zhou, Local boundedness and H\"{o}lder continuity for the parabolic fractional $p$-Laplace equations, Calc. Var. Partial Differential Equations, 60 (1) (2021) 45pp.

\bibitem{GK24} P. Garain and J. Kinnunen, On the regularity theory for mixed local and nonlocal quasilinear parabolic equations, Ann. Sc. Norm. Super. Pisa Cl. Sci. 25 (1) (2024) 495--540.

\bibitem{GV06} U. Gianazza and V. Vespri, A Harnack inequality for solutions of doubly nonlinear parabolic equations, J. Appl. Funct. Anal. 1 (3) (2006) 271--284.	

\bibitem{GM86} M. Giaquinta and G. Modica, Remarks on the regularity of the minimizers of certain degenerate functionals, Manuscripta Math. 57 (1) (1986) 55--99.

\bibitem{G03} E. Giusti, Direct Methods in the Calculus of Variations, World Scientific, Singapore, 2003.

\bibitem{KK07} J. Kinnunen and T. Kuusi, Local behaviour of solutions to doubly nonlinear parabolic equations, Math. Ann. 337 (3) (2007) 705--728.

\bibitem{KL06} J. Kinnunen and P. Lindqvist, Pointwise behaviour of semicontinuous supersolutions to a quasilinear parabolic equation, Ann. Mat. Pura Appl. 185 (3) (2006) 411--435. 

\bibitem{LM18} G. Leugering and G. Mophou, Instantaneous optimal control of friction dominated flow in a gas-network, Shape optimization, Homogenization and optimal control, Springer, Cham, 2018.

\bibitem{L24} N. Liao, H\"{o}lder regularity for parabolic fractional $p$-Laplacian, Calc. Var. Partial Differential Equations, 63 (1) (2024) 34pp.

\bibitem{L21} N. Liao, Regularity of weak supersolutions to elliptic and parabolic equations: lower semicontinuity and pointwise behavior, J. Math. Pures Appl. 147 (9) (2021) 179--204.

\bibitem{M76} M. Mahaffy, A three-dimensional numerical model of ice sheets: tests on the barnes ice cap, northwest territories, J. Geophys. Res. 81 (6) (1976) 1059--1066.

\bibitem{M23} M. Misawa, Expansion of positivity for doubly nonlinear parabolic equations and its application, Calc. Var. Partial Differential Equations, 62 (9) (2023) 48pp.

\bibitem{MN23} M. Misawa and K. Nakamura, Existence of a sign-changing weak solution to doubly nonlinear parabolic equations, J. Geom. Anal. 33 (1) (2023) 44pp.

\bibitem{MN23-1} M. Misawa and K. Nakamura, Intrinsic scaling method for doubly nonlinear parabolic equations and its application, Adv. Calc. Var. 16 (2) (2023) 259--297.

\bibitem{MSS23} K. Moring, L. Sch\"{a}tzler and C. Scheven, Higher integrability for singular doubly nonlinear systems, arXiv:2312.04220v1.

\bibitem{N23} K. Nakamura, Harnack's estimate for a mixed local-nonlocal doubly nonlinear parabolic equation, Calc. Var. Partial Differential Equations, 62 (2) (2023) 45pp.

\bibitem{N22} K. Nakamura, Local boundedness of a mixed local-nonlocal doubly nonlinear equation, J. Evol. Equ. 22 (3) (2022) 38pp.

\bibitem{O96} F. Otto, $L^1$-contraction and uniqueness for quasilinear elliptic-parabolic equations, J. Differential Equations, 131 (1) (1996) 20--38.

\bibitem{P24} H. Prasad, On the weak Harnack estimate for nonlocal equations, Calc. Var. Partial Differential Equations, 63 (3) (2024) 19pp.

\bibitem{SZ23} B. Shang and C. Zhang, Harnack inequality for mixed local and nonlocal parabolic $p$-Laplace equations, J. Geom. Anal. 33 (4) (2023) 24 pp.

\bibitem{SZ22} B. Shang and C. Zhang, H\"{o}lder regularity for mixed local and nonlocal $p$-Laplace parabolic equations, Discrete Contin. Dyn. Syst. 42 (12) (2022) 5817--5837.

\bibitem{S19} M. Str\"{o}mqvist, Local boundedness of solutions to non-local parabolic equations modeled on the fractional $p$-Laplacian, J. Differential Equations, 266 (12) (2019) 7948--7979. 

\bibitem{V06} J.-L. V\'{a}zquez, Smoothing and Decay Estimates for Nonlinear Diffusion Equations: Equations of Porous Medium Type, Oxford Lecture Ser. Math. Appl. Oxford, 2006.

\bibitem{V16} J.-L. V\'{a}zquez, The Dirichlet problem for the fractional $p$-Laplacian evolution equation, J. Differential Equations, 260 (7) (2016) 6038--6056. 

\bibitem{V94} V. Vespri, Harnack type inequalities for solutions of certain doubly nonlinear parabolic equations, J.~Math. Anal. Appl. 181 (1) (1994) 104--131. 

\bibitem{V92} V. Vespri, On the local behaviour of solutions of a certain class of doubly nonlinear parabolic equations, Manuscripta Math. 75 (1) (1992) 65--80.

\bibitem{VV22} V. Vespri and M. Vestberg, An extensive study of the regularity of solutions to doubly singular equations, Adv. Calc. Var. 15 (3) (2022) 435--473.

\end{thebibliography}
\end{document}